\documentclass[11pt,a4paper] {article}
 \usepackage{a4}
\usepackage{amsmath} 
\usepackage{amssymb}
\usepackage{amsthm}
\usepackage{graphicx}
\usepackage{tikz}
\usepackage{float}

\usepackage{color}

\def\+{\oplus}

\newcommand{\HH}{{\mathbb H}}

\newcommand{\R}{{\mathbb R}}

\newcommand{\Z}{{\mathbb Z}}

\newcommand{\Sp}{{\mathbb S}}

\newcommand{\cG}{{\mathcal G}}

\newcommand{\cJ}{{\mathcal J}}

\newcommand{\cC}{{\mathcal C}}

\newcommand{\cM}{{\mathcal M}}

\newcommand{\cT}{{\mathcal T}}

\newcommand{\cR}{{\mathcal R}}

\renewcommand{\epsilon}{\varepsilon}

\renewcommand{\l}{\lambda}

\newcommand{\FL}{{\rm{FL}}}

\newcommand{\ds}{\displaystyle}
\def\squareforqed{\hbox{\rlap{$\sqcap$}$\sqcup$}}
\def\qed{\ifmmode\else\unskip\quad\fi\squareforqed}
\def\smartqed{\def\qed{\ifmmode\squareforqed\else{\unskip\nobreak\hfil
\penalty50\hskip1em\null\nobreak\hfil\squareforqed
\parfillskip=0pt\finalhyphendemerits=0\endgraf}\fi}}

 \newtheorem{theorem}{\textbf{Theorem}}[section]
 \newtheorem{remark}[theorem]{\textbf{Remark}}%[section]
 \newtheorem{lemma}[theorem]{\textbf{Lemma}}%[section]

 \newtheorem{corollary}[theorem]{\textbf{Corollary}}%[section]
 \newtheorem{proposition}[theorem]{\textbf{Proposition}}%[section]
 \newtheorem{definition}[theorem]{\textbf{Definition}}%[section]
\numberwithin{equation}{section}
\title{Homogenization of a transmission problem \\ with Hamilton-Jacobi equations and  a two-scale interface.\\ Effective transmission conditions} % in optimal control
\author{Yves Achdou \thanks { Univ. Paris Diderot, Sorbonne Paris Cit{\'e}, Laboratoire Jacques-Louis Lions, UMR 7598, UPMC, CNRS, F-75205 Paris, France.
 achdou@ljll.univ-paris-diderot.fr},
Nicoletta Tchou \thanks {IRMAR, Universit{\'e} de Rennes 1, Rennes, France, nicoletta.tchou@univ-rennes1.fr}
}
\begin{document}

\maketitle
\begin{abstract}
  We consider a family of optimal control problems in the plane  with  dynamics and running costs possibly discontinuous across a two-scale oscillatory interface.
Typically, the amplitude of the oscillations is of the order of $\epsilon$ while the period is of the order of  $\epsilon^ 2$.
 As $\epsilon\to 0$, the  interfaces  tend to a straight line $\Gamma$.
  We study the asymptotic behavior of the value function as $\epsilon\to 0$.
 We prove that the value function tends to the solution of  Hamilton-Jacobi equations in the two half-planes limited by  $\Gamma$, 
with an effective transmission condition on $\Gamma$ keeping track of the oscillations.
\end{abstract}

\section{Introduction}
\label{sec:setting}

The main motivation of  this paper is to study the asymptotic behavior as $\epsilon\to 0$ of the value function of an optimal 
control problem in $\R^2$  in which the running cost and dynamics may jump across a periodic oscillatory interface $\Gamma_{\epsilon,\epsilon}$, when 
the oscillations of $\Gamma_{\epsilon,\epsilon}$ have an amplitude   of the order of $\epsilon$ and  a period  of the order of $\epsilon^2$, (see Figure \ref{fig:geom1} below, which actually describes a more general case).
The respective roles of the two indices in $\Gamma_{\epsilon,\epsilon}$  will be explained in \S~\ref{sec:geometry} below.
 The interface $\Gamma_{\epsilon,\epsilon}$ separates two unbounded regions of $\R^2$, $\Omega_{\epsilon,\epsilon}^L$ and $\Omega_{\epsilon,\epsilon}^R$. 
The present work is a natural continuation of a previous one, \cite{MR3565416}, in which both the amplitude and the period of the oscillations were of the order of $\epsilon$.
In \cite{MR3565416}, it was possible to make a change of variables in order to map the interface onto a flat one.  Here, it is no longer possible, 
and the route leading to the homogenization result becomes more complex.
\\
To characterize the optimal control problem,  one has to specify the admissible dynamics at a point $x\in \Gamma_{\epsilon,\epsilon}$: in our setting, no mixture is allowed at the interface, 
i.e. the admissible dynamics are the ones corresponding to the subdomain $ \Omega_{\epsilon,\epsilon}^L$ {\bf and} entering $ \Omega_{\epsilon,\epsilon}^L$,
 or corresponding to the subdomain $ \Omega_{\epsilon,\epsilon}^R$ {\bf and} entering $ \Omega_{\epsilon,\epsilon}^R$. 
Hence the situation differs from those studied in the 
articles of G. Barles, A. Briani and E. Chasseigne \cite{barles2011bellman,barles2013bellman} and of  G. Barles, A. Briani, E. Chasseigne and N. Tchou \cite{MR3424272}, 
in which mixing is allowed at the interface. The optimal control problem under consideration has been first studied in \cite{oudet2014}:  the value function is
characterized as the viscosity solution of a Hamilton-Jacobi equation
 with special transmission conditions on $ \Gamma_{\epsilon,\epsilon}$; a comparison principle for this problem is proved in \cite{oudet2014} with arguments 
 from the theory of optimal control similar to those introduced in \cite{barles2011bellman,barles2013bellman}. In parallel to \cite{oudet2014}, 
Imbert and Monneau have studied similar problems
from the viewpoint of PDEs, see \cite{imbert:hal-01073954},  and have obtained comparison  results for quasi-convex Hamiltonians.
There has been a very active research effort on finding simpler and more general/powerful proofs of the above-mentioned comparison results,
 see  \cite{2016arXiv161101977B} and the very recent work of P-L. Lions and P. Souganidis \cite{2017arXiv170404001L}.
\\
In particular,  \cite{imbert:hal-01073954} contains a characterization of the viscosity solution of the transmission problem
 with a reduced set of test-functions;  this characterization will be used in the present work.
Note that \cite{oudet2014,imbert:hal-01073954} can be seen as extensions of articles 
devoted to the analysis of Hamilton-Jacobi equations on networks, see \cite{MR3057137,MR3023064,MR3358634,MR3621434,MR3556345}, 
because the notion of interface used there can be seen as a generalization of the notion of vertex (or junction) for a network.
\\
We will see that  as $\epsilon$ tends to $0$, the value function converges  to the solution of an effective problem
related to a flat interface $\Gamma$,  with Hamilton-Jacobi equations in the  half-planes limited by $\Gamma$ and a transmission condition on $\Gamma$.
Whereas the partial differential equation  far from the interface is unchanged, the main difficulty consists in finding the effective transmission condition on $\Gamma$.
Naturally, the latter depends on the dynamics and running costs  but also 
keeps  memory of the vanishing oscillations. 
The present work is strongly related to \cite{MR3565416}, but also to two articles, \cite{MR3299352} and \cite{MR3441209},
  about singularly  perturbed problems  leading to effective Hamilton-Jacobi equations on networks.
In \cite{MR3299352},  the authors of the present paper study a family of star-shaped planar domains  $D^\epsilon$ 
made of $N$ non intersecting semi-infinite strips  of thickness $\epsilon$ and of a central region whose diameter is proportional to $\epsilon$.  
As  $\epsilon \to 0$,  the domains $D^\epsilon$ tend to a  network $\cG$ made of $N$ half-lines sharing an endpoint $O$, named the vertex or junction point.
 For infinite horizon optimal control problems  in which the state is constrained to remain in the closure of $D^\epsilon$, 
the value function tends to the solution of a Hamilton-Jacobi equation on $\cG$, with an effective  transmission condition at $O$. 
 The related effective Hamiltonian, which  corresponds to  trajectories staying close to   the junction point, was 
obtained in \cite{MR3299352} as the limit of a sequence of ergodic constants corresponding to larger and larger bounded subdomains. 
Note that the same problem  and the question of the correctors in unbounded domains were also discussed by P-L. Lions in his lectures at Coll{\`e}ge de France respectively in January 2017, and
in  January and February 2014, see \cite{PLL}.
  The same kind of  construction was then used in \cite{MR3441209}, in which  Galise, Imbert and Monneau study a family of time dependent Hamilton-Jacobi equations 
in a simple network composed of two half-lines with a perturbation of the Hamiltonian localized in a small region close to the junction.  
In  \cite{MR3441209}, a key point was the use of a single test-function at the vertex which was first proposed in \cite{MR3621434,imbert:hal-01073954}.
 This idea will be also used in the present work. Note  that similar techniques were used in the  recent works of Forcadel et al, \cite{forcadel:hal-01097085,MR3640560,forcadel:hal-01332787}, which deal with applications to traffic flows. Finally, multiscale homogenization and singular perturbation problems with first and second order Hamilton Jacobi equations  (without discontinuities)
have been addressed in \cite{MR2371792,MR2487745}.
\\
Note that slight modifications of the techniques used below yield the asymptotic behavior of the transmission problems with oscillatory  interfaces of amplitude $\epsilon$ and period $\epsilon ^{1+q}$ with $q\ge 0$, (see \S~\ref{sec:effect-probl-obta} and \cite{MR3565416} for $q=0$ and Remark \ref{sec:main-result-5} below for $q>0$).
Also, even if we focus on  a two-dimensional problem,  all the results below hold in the case when $\R^N$ is divided into two subregions, separated by a smooth and periodic $N-1$ dimensional oscillatory interface with two scales. 
Finally, we wish to stress the fact that  an important possible application of our work is the homogenization of a transmission problem  in geometrical optics, with two media separated by a two-scale interface.
%Our methodology will be led by the intuitive idea that since the two small scales $\epsilon^2$ and $\epsilon$ are well separate
\\
The paper is organized as follows: in the remaining part of \S~\ref{sec:setting}, we set the problem. We will see in particular that it is convenient to consider a more general setting 
than the one  described above, with two small parameters $\eta $ and $\epsilon$ instead of one: more precisely, the amplitude of the oscillations will be of the order of $\eta$ whereas the period will of the order of $\eta\epsilon$. In \S~\ref{sec:effect-probl-obta}, we keep $\eta$ fixed  while $\epsilon$ tends to $0$: the region where the two media are mixed is a strip whose width is of  the order of $\eta$: in this region, an effective Hamiltonian is found by  classical  homogenization techniques, see \cite{LPV}; the main difficulty is to obtain  
the effective transmission conditions on the  boundaries of the strip (two parallel straight lines) and to prove
 the convergence. The techniques will be reminiscent of \cite{MR3299352,MR3441209,MR3565416}, because only one parameter tends to $0$. The effective transmission conditions keeps track of the geometry of the interface at the scale $\epsilon$. \\
 In \S~\ref{sec:second-passage-limit}, we take the latter effective problem which depends on $\eta$, and have $\eta$ tend to $0$: we obtain a new effective transmission condition on a single flat interface, and prove the convergence result. Note that this passage to the limit is an intermediate step in order to study the two-scale homogenization problem described in the beginning of the introduction, but that it has also an interest for itself.
\\
  In \S~\ref{sec:simult-pass-limit}, we take $\eta=\epsilon$, i.e. we consider the 
interface $\Gamma_{\epsilon,\epsilon}$ described at the beginning of the introduction, and let $\epsilon$ tend to $0$: at the limit, we obtain the same effective problem as the one found in \S~\ref{sec:second-passage-limit}, by letting first $\epsilon$  then $\eta$ tend to $0$. 
\\
Sections~\ref{sec:effect-probl-obta}, \ref{sec:second-passage-limit} and \ref{sec:simult-pass-limit} are organized in the same way: the main result is stated first, then proved in the remaining part of the section. For the conciseness of \S~\ref{sec:effect-probl-obta}, some technical proofs will be given in an appendix.

\subsection{The geometry}
\label{sec:geometry}
Let $(e_1, e_2)$ be an orthonormal basis of $\R^2$.
For two real numbers $a,b$ such that $0<a<b<1$, consider the  set $\Sp=\left\{ a, b \right\}+\Z$.
Let $g:  \R \to \R$ be a continuous function,  periodic with period $1$, such that
\begin{enumerate}
\item $g$ is $\cC^2$ in $\R \backslash \Sp$
\item $g\left(a \right)= g\left(b\right)=0$
\item $\ds \mathop {\lim}_{t\to a^-} g'(t)=\mathop{\lim}_{t\to a^+} g'(t)=+\infty$ and 
 $\ds \mathop {\lim}_{t\to a^-} \frac {g''(t)}{g'(t)}=\mathop{\lim}_{t\to a^+}  \frac {g''(t)}{g'(t)}=0$
\item $\ds \mathop {\lim}_{t\to  b^-} g'(t)=\mathop{\lim}_{t\to b^+} g'(t)=-\infty$ and 
 $\ds \mathop {\lim}_{t\to  b^-}  \frac {g''(t)}{g'(t)}=\mathop{\lim}_{t\to b^+}  \frac {g''(t)}{g'(t)}=0$               
\end{enumerate}
Let $G$ be the multivalued Heavyside step function, periodic with period $1$,  such that
\begin{enumerate}
\item $G(a)=G(b)=[-1,1]$
\item $G(t)=\{1\}$ if $t\in (a,b )$
\item $G(t)=\{-1\}$ if $t\in [0,a)\cup(b,1]$
\end{enumerate}
Let $\eta$ and $\epsilon$ be two positive parameters: consider the $\cC^2$ curve $\Gamma_{\eta,\epsilon}$ defined as the graph of the multivalued function
 $g_{\eta,\epsilon}: x_2\mapsto \eta G(\frac {x_2}{ \epsilon \eta})+ \eta\epsilon  g(\frac {x_2}{ \epsilon \eta})$. 
We also define the domain $\Omega_{\eta,\epsilon}^R $  (resp. $\Omega_{\eta,\epsilon}^L $) as the epigraph (resp. hypograph) 
of  $g_{\eta,\epsilon}$:
\begin{eqnarray}\label{eq:44}
  \Omega_{\eta,\epsilon}^R= & \{ x\in \R^2 :   x_1> g_{\eta,\epsilon} (x_2)\},\\
\label{eq:45}
  \Omega_{\eta,\epsilon}^L= & \{ x\in \R^2 :   x_1< g_{\eta,\epsilon} (x_2)\}.
\end{eqnarray}
The unit normal vector $n_{\eta,\epsilon}(x)$ at $ x\in \Gamma_{\eta,\epsilon}$ is defined  as follows: setting $y_2= \frac { x_2}{\eta \epsilon}$,
\begin{displaymath} 
n_{\eta,\epsilon}(x)= \left\{
  \begin{array}[c]{cl}
    \ds  
\left (1 + \left( g'(y_2)\right)^2 \right) ^{-1/2}  \left(  e_1 -   g' (y_2)   e_2\right)
  \quad &\hbox{ if    }\quad y_2 \notin  \Sp\\ 
\ds - e_2     \quad&\hbox{ if    }\quad y_2 = a \mod{1}\\

\ds  e_2     \quad&\hbox{ if    } \quad y_2 = b \mod{1}.
  \end{array}
\right.
\end{displaymath}
Note that $n_{\eta,\epsilon}(x)$ is  oriented from $\Omega^L_{\eta,\epsilon}$ to $\Omega^R_{\eta,\epsilon}$.

\begin{figure}[H]
  \begin{center}
    \begin{tikzpicture}[scale=0.5, trans/.style={thick,<->,shorten >=2pt,shorten <=2pt,>=stealth} ]
      \draw[red,thick] (0,10) -- (10,10);
      \draw[red,thick] (10,10) .. controls (11,10) and (11,9.5) .. (10,9.5);
      \draw[red,thick] (10,9.5) -- (0,9.5);
      \draw[red,thick] (0,9.5) .. controls (-0.5,9.5) and (-0.5,9) .. (0,9);
      \draw[red,thick] (0,9) -- (10,9);
      \draw[red,thick] (10,9) .. controls (11,9) and (11,8.5) .. (10,8.5);
      \draw[red,thick] (10,8.5) -- (0,8.5);
      \draw[red,thick] (0,8.5) .. controls (-0.5,8.5) and (-0.5,8) .. (0,8);
      \draw[red,thick] (0,8) -- (10,8);
      \draw[red,thick] (10,8) .. controls (11,8) and (11,7.5) .. (10,7.5);
      \draw[red,thick] (10,7.5) -- (0,7.5);
      \draw[red,thick] (0,7.5) .. controls (-0.5,7.5) and (-0.5,7) .. (0,7);
      \draw[red,thick] (0,7) -- (10,7);
      \draw[red,thick] (10,7) .. controls (11,7) and (11,6.5) .. (10,6.5);
      \draw[red,thick] (10,6.5) -- (0,6.5);
      \draw[red,thick] (0,6.5) .. controls (-0.5,6.5) and (-0.5,6) .. (0,6);
      \draw[red,thick] (0,6) -- (10,6);
      \draw[red,thick] (10,6) .. controls (11,6) and (11,5.5) .. (10,5.5);
      \draw[red,thick] (10,5.5) -- (0,5.5);
      \draw[red,thick] (0,5.5) .. controls (-0.5,5.5) and (-0.5,5) .. (0,5);
      \draw[red,thick] (0,5) -- (10,5);
      \draw[red,thick] (10,5) .. controls (11,5) and (11,4.5) .. (10,4.5);
      \draw[red,thick] (10,4.5) -- (0,4.5);
      \draw[red,thick] (0,4.5) .. controls (-0.5,4.5) and (-0.5,4) .. (0,4);
      \draw[red,thick] (0,4) -- (10,4);
      \draw[red,thick] (10,4) .. controls (11,4) and (11,3.5) .. (10,3.5);
      \draw[red,thick] (10,3.5) -- (0,3.5);
      \draw[red,thick] (0,3.5) .. controls (-0.5,3.5) and (-0.5,3) .. (0,3);
      \draw[trans] (0.1,11) -- (10.1,11) ;
      \draw (5,11)  node[above]{$2\eta$};
      \draw[trans] (9.7,11) -- (10.8,11) ;
      \draw (10.4,11)  node[above]{{\small $\sim \eta\epsilon$}};
      \draw[trans] (-0.6,11) -- (0.4,11) ;
      \draw (-0.2,11)  node[above]{{\small $\sim \eta\epsilon$}};
      \draw[trans] (5,10.1) -- (5,8.9) ;
      \draw (5,9.5)  node[right]{$\eta\epsilon$};
      \draw (-1,6.5)  node[left]{$\Omega^L_{\eta,\epsilon}$};
      \draw (11,6.5)  node[right]{$\Omega^R_{\eta,\epsilon}$};
    \end{tikzpicture}
    \caption{The oscillatory interface $\Gamma_{\eta,\epsilon}$ separates $\Omega^L_{\eta,\epsilon}$ and $\Omega^R_{\eta,\epsilon}$. It has two scales: its amplitude $\eta$ and period $\eta\epsilon$}
    \label{fig:geom1}
  \end{center}
\end{figure}
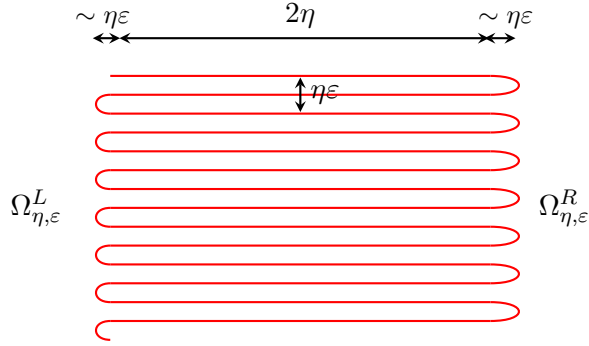
In \S~\ref{sec:effect-probl-obta}, we will let $\epsilon$ tend to zero and keep $\eta$ fixed. 
In \S~\ref{sec:simult-pass-limit}, we will focus on the case when $\eta=\epsilon$ and let $\epsilon$ tend to $0$. 

\subsection{The optimal control problem in  $\Omega^L_{\eta,\epsilon}\cup \Omega^R_{\eta,\epsilon} \cup  \Gamma_{\eta,\epsilon}$}
\label{optimal1}
We consider infinite-horizon optimal control problems which have different dynamics and running costs in the regions 
$\Omega^i_{\eta,\epsilon}$, $i=L, R$.
  The sets of controls associated to the index $i=L,R$ will be called $A^i$; 
similarly, the notations $f^i$ and  $\ell^i$ will be used for the dynamics and running costs.
The following  assumptions will be made in all the present work.
\subsubsection{Standing Assumptions}
\label{sec:assumptions}
\begin{description}
\item{[H0]} $A$ is a metric space (one can take $A=\R^m$). For $i=L,R$,  $A^i$ is a non empty compact subset of $A$ and
 $f^i:  A^i \to \R^2$ is  a continuous function. The sets $A^i$ are disjoint.
 Define  $M_f=   \max_{i=L, R}   \sup_{
  a \in    A^i } | f^i(a)| $.
  The notation $F^i$ will be used for the set $F^i=\{f^i(a), a\in A^i\} $.
\item{[H1]} For $i=L,R$, the function $\ell^i:  A^i\to \R$ is  continuous and bounded. 
 Define  $M_\ell=   \max_{i=L, R}   \sup_{
%x\in \R^2, 
a \in    A^i } | \ell^i(a)| $.
\item{[H2] }  For any $i=L,R$, the non empty  set $\FL^i= \{ (f^i(a), \ell^i(a) ) , a\in A^i\}$ is  closed and convex. 
\item{[H3] } There is a real number $\delta_0>0$ such that for $i=L,R$,
$   B(0,\delta_0) \subset F^i$.
\end{description}
We stress the fact that all the  results below hold provided  the latter  assumptions are satisfied,  
although, in order to avoid tedious repetitions, we will not mention them explicitly in the statements.
\begin{remark}\label{sec:standing-assumptions}
  We have assumed that the dynamics $f^i$ and running costs $\ell^i$, $i=L,R$, do not depend on $x$. This assumption is made only for simplicity. With further classical assumptions, it would be possible to generalize all the results contained in this paper to the case
when $f^i$ and $\ell^i$ depend on $x$: typical such assumptions are
\begin{enumerate}
\item the Lipschitz continuity of $f^i$ with respect to $x$ uniformly in $a\in A^i$: there exists $L_f$ such that
for $i=L,R$,  $\forall a\in A^i$, $x,y\in \R^2$, $   |f^i(x,a)-f^i(y,a)|\le L_f |x-y|$
\item  the existence of  a modulus of continuity $\omega_\ell$ 
 such that for any $i=L,R$, $x,y\in \R^2$ and $a\in A^i$, 
 \begin{displaymath}
 |\ell^i(x,a)-\ell^i(y,a)|\le \omega_{\ell} (|x-y|).
 \end{displaymath}
\end{enumerate}
Even if these assumptions are standard, keeping track of a possible slow dependency of the Hamiltonian with respect to $x$ in the homogenization process below would have led us to tackle several technical questions and to significantly increase the length of the paper. It would have also made the 
essential ideas more difficult to grasp.\\
Moreover, it is clear that if $f^i $ and $\ell^i$ do depend on $x$ except in a strip containing the oscillatory interface,  
for example the strip $\{ x: |x_1|<1 \}$ for $\epsilon$ and $\eta$ small enough, then all what follows holds and does not require any further technicality.
\end{remark}

\subsubsection{The optimal control problem}
\label{sec:optim-contr-probl}
Let the closed set $\cM_{\eta,\epsilon}$ be defined as follows:
\begin{equation}
  \label{eq: def-M_epsilon}
  \cM_{\eta,\epsilon}=\left\{(x,a);\; x\in \R^2,\quad  a\in A^i  \hbox{ if } x\in \Omega_{\eta,\epsilon}^i,\; i=L,R,\;  \hbox{ and } a     \in A^L\cup A^R  \hbox{ if } x \in \Gamma_{\eta,\epsilon}\right \}.
  \end{equation}
  The dynamics $f_{\eta,\epsilon}$ is a function defined in  $\cM_{\eta,\epsilon}$ with values in $\R^2$: 
\begin{displaymath}
\forall (x,a)\in \cM_{\eta,\epsilon},\quad\quad   f_{\eta,\epsilon}(x, a)= f^i(a)
\quad \hbox{ if } x\in \Omega_{\eta,\epsilon}^i \hbox{ or }(x\in\Gamma_{\eta,\epsilon} \hbox{ and }     a\in A^i).
\end{displaymath}
The function $f_{\eta,\epsilon}$ is continuous on $\cM_{\eta,\epsilon}$ because the sets $A^i$ are disjoint. Similarly, let the running cost $\ell_{\eta,\epsilon}: \cM_{\eta,\epsilon}\to \R$ be given by
\begin{displaymath}
\forall (x,a)\in \cM_{\eta,\epsilon},\quad\quad   \ell_{\eta,\epsilon}(x, a)=
 \ell^i(a).
 \quad \hbox{ if } x\in \Omega_{\eta,\epsilon}^i \hbox{ or }(x\in\Gamma_{\eta,\epsilon} \hbox{ and }     a\in A^i).
\end{displaymath}
For $x\in \R^2$, the set of admissible  trajectories starting from $x$ is 
\begin{equation}
  \label{eq:2}
\cT_{x,\eta,\epsilon}=\left\{
  \begin{array}[c]{ll}
    ( y_x, a)  \ds \in L_{\rm{loc}}^\infty( \R^+; \cM_{\eta,\epsilon}): \quad   & y_x\in {\rm{Lip}}(\R^+; \R^2), 
 \\  &\ds  y_x(t)=x+\int_0^t f_{\eta,\epsilon}( y_x(s), a(s)) ds \quad  \forall t\in \R^+ 
  \end{array}\right\}.
\end{equation}
The cost associated to the trajectory $ ( y_x, a)\in \cT_{x,\eta,\epsilon}$ is 
\begin{equation}
  \label{eq:46}
  \cJ_{\eta,\epsilon}(x;( y_x, a) )=\int_0^\infty \ell_{\eta,\epsilon}(y_x(t),a(t)) e^{-\lambda t} dt,
\end{equation}
with $\lambda>0$. The value function of the infinite horizon optimal control problem is 
\begin{equation}
  \label{eq:4}
v_{\eta,\epsilon}(x)= \inf_{( y_x, a)\in \cT_{x,\eta,\epsilon}}   \cJ_{\eta,\epsilon}(x;( y_x, a) ).
\end{equation}
\begin{proposition}\label{sec:assumption}
 The value function $v_{\eta,\epsilon}$ is bounded uniformly in $\eta$ and $\epsilon$ and continuous in $ \R^2$.
\end{proposition}
\begin{proof}
This result is classical and can be proved with the same arguments as in \cite{MR1484411}.
\end{proof}

\subsection{The Hamilton-Jacobi equation}
\label{sec:hamilt-jacobi-equat}
Similar optimal control problems have recently been studied in \cite{MR3358634,MR3621434,oudet2014,imbert:hal-01073954}.
It turns out that $v_{\eta,\epsilon}$ can be characterized as the viscosity solution of a Hamilton-Jacobi equation with a discontinuous Hamiltonian,
(once the notion of viscosity solution has been specially tailored to cope with the above mentioned discontinuity).  
We briefly recall the definitions used e.g. in \cite{oudet2014}.
\paragraph{Hamiltonians} For $i=L,R$, let the Hamiltonians $H^i: \R^2\rightarrow \R $  and  $H_{ \Gamma_{\eta,\epsilon}}:\Gamma_{\eta,\epsilon}\times \R^2\times \R^2\to \R$  be defined by
\begin{eqnarray}
  \label{eq:7}
H^i(p)&=& \max_{a\in A^i} (-p \cdot f^i(a) -\ell^i(a)),\\
  \label{eq:8}
H_{ \Gamma_{\eta,\epsilon}} (x,p^L,p^R)&=& \max \{ \;H_{ \Gamma_{\eta,\epsilon}}^{+,L}(x,p^L),H_{ \Gamma_{\eta,\epsilon}}^{-,R} (x,p^R)\},
\end{eqnarray}
where in (\ref{eq:8}),  $p^L\in \R^2 $ and $p^R \in \R^2$.

\begin{eqnarray}
  \label{eq:29}
H_{ \Gamma_{\eta,\epsilon}}^{-,i} (x,p)= \max_{a\in A^i \hbox{ s.t. }   f^i(a)\cdot n_{\eta,\epsilon}(x)\ge 0} (-p\cdot f^i(a) -\ell^i(a)), \quad \forall x\in \Gamma_{\eta,\epsilon}, \forall p\in \R^2,
\\
\label{eq:30}
H_{ \Gamma_{\eta,\epsilon}}^{+,i} (x,p)= \max_{a\in A^i \hbox{ s.t. }   f^i(a)\cdot n_{\eta,\epsilon}(x)\le 0} (-p\cdot f^i(a) -\ell^i(a)), \quad \forall x\in \Gamma_{\eta,\epsilon}, \forall p\in \R^2.
\end{eqnarray}

\paragraph{Test-functions}
For $\eta >0$ and $\epsilon>0$,  the function $\phi: \R^2\to \R$ is an admissible test-function if
$\phi$ is continuous in $\R^2$  and for any $i\in \{L,R\}$, $\phi|_{\overline{\Omega^i_{\eta,\epsilon}}} \in\cC^1(\overline{\Omega^i_{\eta,\epsilon}})$.
\\The set of admissible test-functions is noted $\cR_{\eta,\epsilon}$. If $\phi \in \cR_{\eta,\epsilon}$, $x\in \Gamma_{\eta,\epsilon}$ and $i\in \{L,R\}$, we  set
$\displaystyle D\phi^i(x)= \lim_{\overset{ x'\to x}{x'\in \Omega^i_{\eta,\epsilon}}}D\phi(x')$.
\begin{remark}
  \label{sec:test-functions}
If $x\in  \Gamma_{\eta,\epsilon}$, $\phi$ is test-function and $p^L= D\phi^L (x)$, $p^R= D\phi^R (x)$, then $p^L -p^R$ is colinear to $n_{\eta,\epsilon}(x)$  defined in \S~\ref{sec:geometry}.
\end{remark}

\paragraph{Definition of viscosity solutions}
We are going to define viscosity solutions of the following transmission problem:
 \begin{eqnarray}
\label{eq:58}
\lambda u(x)+H^L(Du(x))&= 0, \quad \quad &    \hbox{ if }x\in \Omega^L_{\eta,\epsilon},\\
\label{eq:59}
\lambda u(x)+H^R(Du(x))&= 0, \quad \quad &    \hbox{ if }x\in \Omega^R_{\eta,\epsilon},\\
\label{eq:60}
\lambda u(x)+H_{ \Gamma_{\eta,\epsilon}} (x,Du^L(x),Du^R(x))&= 0, \quad \quad &    \hbox{ if }x\in\Gamma_{\eta,\epsilon},
 \end{eqnarray}
where $u^L$ (respectively $u^R$) stands for $u|_{\overline{\Omega^L_{\eta,\epsilon}}}$ ( respectively $u|_{\overline{\Omega^R_{\eta,\epsilon}}}$). For brevity, we  also note this problem 
  \begin{equation}\label{HJaepsilon}
    \lambda u+{\cal{H}}_{\eta,\epsilon}(x, Du)=0.
\end{equation}

\begin{itemize}
\item An upper semi-continuous function $u:\R^2\to\R$ is a subsolution of \eqref{HJaepsilon}
 if for any $x\in \R^2$, any $\phi\in\cR_{\eta,\epsilon}$ s.t. $u-\phi$ has a local maximum point at $x$, then
 \begin{eqnarray}
  \label{eq:5bis}
\lambda u(x)+H^i(D\phi^i(x))&\le 0, \quad \quad &    \hbox{ if }x\in \Omega^i_{\eta,\epsilon},\\
\label{eq:5bisgamma}
\lambda u(x)+H_{ \Gamma_{\eta,\epsilon}} (x,D\phi^L(x),D\phi^R(x))&\le 0, \quad \quad &    \hbox{ if }x\in\Gamma_{\eta,\epsilon},
 \end{eqnarray}
where,  for  $x\in \Gamma_{\eta,\epsilon}$, the notation $D\phi^i(x)$ is introduced in the definition of the test-functions,
see also Remark~\ref{sec:test-functions}.
\item  A lower semi-continuous function $u:\R^2\to\R$ is a supersolution of \eqref{HJaepsilon}
 if for any $x\in \R^2$, any $\phi\in\cR_{\eta,\epsilon}$ s.t. $u-\phi$ has a local minimum point at $x$, then
 \begin{eqnarray}
  \label{eq:6bis} \lambda u(x)+H^i(D\phi^i(x))&\geq 0,  \quad \quad &    \hbox{ if }x\in \Omega^i_{\eta,\epsilon},\\
\label{eq:6bisgamma} \lambda u(x)+H_{ \Gamma_{\eta,\epsilon}} (x,D\phi^L(x),D\phi^R(x))&\ge 0 \quad \quad &    \hbox{ if }x\in\Gamma_{\eta,\epsilon}.
 \end{eqnarray}
  \item A continuous function $u:\R^2\to\R$ is a   viscosity solution of \eqref{HJaepsilon}   if it is both a viscosity sub and supersolution   of \eqref{HJaepsilon}.
\end{itemize}
We skip the proof of the following theorem, see \cite{oudet2014,imbert:hal-01073954}.
\begin{theorem}
  \label{existence-epsilon_a}
The value function $v_{\eta,\epsilon}$  defined in \eqref{eq:4}  is the unique bounded  viscosity solution of \eqref{HJaepsilon}.
\end{theorem}

\subsection{The main result and the general orientation}
\label{sec:main-result-main}
We set 
\begin{equation}
  \label{eq:104}
\Omega^L=\{x\in \R^2, x_1<0\}, \quad \Omega^R=\{x\in \R^2, x_1>0\}, \quad  \Gamma=\{x\in \R^2, x_1=0\}.
\end{equation}
\paragraph{Informal statement of the main result}
Our main result, namely Theorem \ref{sec:main-result-4} below,
 is that, as $\epsilon\to 0$, $v_{\epsilon,\epsilon}$ converges locally uniformly  to $v$, the unique bounded viscosity solution of 
  \begin{eqnarray}
    \label{eq:14}
    \lambda v(z)+ H^L( D v(z))  = 0 & &\hbox{if } z\in \Omega^L,\\
    \label{eq:15}
    \lambda v(z)+ H^R( D v(z))  = 0 & &\hbox{if } z\in \Omega^R,\\ 
 \label{eq:17}    \lambda v(z)+\max\left(E(\partial_{z_2}v(z)), H^{L,R}( D v^L(z),D v^R(z)) \right) = 0  & &\hbox{if } z\in \Gamma.
  \end{eqnarray}
The  Hamiltonians $H^L$ and $H^R$ are defined in (\ref{eq:7}).
In the effective transmission condition (\ref{eq:17}),
 \begin{equation}
\label{eq:25}
 H^{L,R}( p^L,p^R) = \max \{ \;H^{+,1,L}(p^L),H^{-,1,R} (p^R)\},
\end{equation}
for   $p^L, p^R \in \R^2 $.  For $i=L,R$,
 $H^{+,1,i}(p)$ (respectively $H^{-,1,i}(p)$) is the nondecreasing (respectively  nonincreasing) part of the Hamiltonian $H^i$ with respect to $p_1$. In what follows, $p^L -p^R$ will be colinear to $e_1$.
 The effective flux-limiter  $E: \R\to \R$ will be characterized in \S \ref{sec:second-passage-limit} below. 
\\
For  brevity, the problem in (\ref{eq:14})-(\ref{eq:17}) will  sometimes be noted
\begin{equation}\label{eq:16}
  \lambda v(z)+{\cal{H}}(z, Dv(z))=0.
\end{equation}
\paragraph{General orientation}
The proof of this result will be done by using Evans' method of perturbed test-functions, see \cite{MR1007533}. Such a method requires to build a family of correctors depending on a single real variable $p_2$ (which stands for the derivative of $v$ along $\Gamma$).
The corrector, that will be noted $\xi_\epsilon(p_2,\cdot)$, solves a cell problem,  see (\ref{eq:48}) below,
 with a transmission condition on the interface $\Gamma_{1,\epsilon}$, (note that the original geometry 
 is dilated by a factor $1/\epsilon$). The  ergodic constant associated to the latter cell problem will be noted $E_\epsilon(p_2)$ in \S~\ref{sec:simult-pass-limit}. The fact that the corrector and the ergodic constant still depend on $\epsilon$ is connected to the existence of two small scales in the problem.

The existence of the pairs  $(\xi_\epsilon(p_2),E_\epsilon(p_2) )$ and the asymptotic behavior of
 $\xi_\epsilon(p_2)$ as $\epsilon\to 0$ will be obtained essentially by using the arguments
 proposed in \S \ref{sec:effect-probl-obta} below. In fact, for a reason that will soon become clear,  in \S \ref{sec:effect-probl-obta}, we consider the interfaces $\Gamma_{\eta,\epsilon}$ instead of $\Gamma_{1,\epsilon}$, for a fixed arbitrary  positive parameter $\eta$. Then, the  region where the two media are mixed is  a strip whose width is $\sim \eta$, see Figure~\ref{fig:geom2}: in this region, an effective Hamiltonian is found by  classical  homogenization techniques; the main achievement of \S  \ref{sec:effect-probl-obta} is to obtain the effective transmission conditions on the  boundaries of the strip (two parallel straight lines) and to prove  the convergence of the solutions of the transmission problems as $\epsilon\to 0$. 
\\
Next, in \S~\ref{sec:second-passage-limit}, we pass to the limit as $\eta\to 0$ in the effective problem that we have just obtained in \S  \ref{sec:effect-probl-obta}. We obtain (\ref{eq:16}) at the limit $\eta\to 0$, whose solution is unique, (the well posedness of (\ref{eq:16}) implies the convergence of the whole family of solutions as $\eta\to 0$) and a characterization of $E(p_2)$ in (\ref{eq:17}), see (\ref{eq:18}) below.

In the proof of the main result stated above and in Theorem~\ref{sec:main-result-4}, 
 Evans' method  requires to study the asymptotic behavior of  $E_\epsilon(p_2)$ and  of $x\mapsto \epsilon \xi_\epsilon(p_2, x/\epsilon)$  as $\epsilon \to 0$. This is precisely what is done in \S~\ref{sec:simult-pass-limit} below, where, in particular, we prove that $\lim_{\epsilon\to 0} E_\epsilon(p_2)= E(p_2)$ given by (\ref{eq:18}).

 Therefore, we will prove that the limit of $v_{\epsilon,\epsilon}$  as $\epsilon\to 0$ can be obtained by
 considering the transmission problems with interfaces $\Gamma_{\eta,\epsilon}$, letting $\epsilon$ tend to $0$ first, then $\eta$ tend to  $0$.
 This is coherent with the intuition that since the two scales $\epsilon^2$ and $\epsilon$ are well separated, 
 the asymptotic behavior can be obtained in two successive steps.

In what follows, a significant difficulty in the construction of the correctors is
the unboundedness of the domains in which they should be defined.
 It is addressed by using the ideas proposed in \cite{MR3299352,MR3441209,MR3565416}.

\section{The effective problem obtained by letting $\epsilon$ tend to $0$}
\label{sec:effect-probl-obta}
In \S~\ref{sec:effect-probl-obta}, $\eta$ is a fixed positive number, whereas $\epsilon$ tends to $0$. 
\subsection{Main result}\label{sec:main-result}
Let the domains $\Omega^{L}_{\eta}$, $\Omega^{M}_{\eta}$ and $\Omega^{R} _{\eta}$ and the straight lines $\Gamma^{L,M}_\eta$,  $\Gamma^{M,R}_\eta$  be defined by 
\begin{eqnarray}
  \label{eq:32}
    \Omega^{L}_{\eta}=\{x\in \R^2, x_1<-\eta\}, \quad   \Omega^{M}_{\eta}=\{x\in \R^2, |x_1|<\eta\},\quad   \Omega^{R}_{\eta}=\{x\in \R^2, x_1>\eta\},\\
\label{eq:33}
    \Gamma^{L,M}_\eta=\{x\in \R^2, x_1=-\eta\}, \quad    \Gamma^{M,R}_\eta=\{x\in \R^2, x_1=\eta\}.
\end{eqnarray}

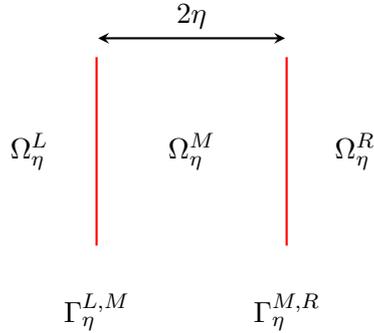
\begin{figure}[H]
  \begin{center}
    \begin{tikzpicture}[scale=0.5, trans/.style={thick,<->,shorten >=2pt,shorten <=2pt,>=stealth} ]
      \draw[red,thick] (0,0) -- (0,5);
      \draw[red,thick] (5,0) -- (5,5);
      \draw[trans] (-0.1,5.5) -- (5.1,5.5) ;
      \draw (2.5,5.5)  node[above]{$2\eta$};
      \draw (-1,2.5)  node[left]{$\Omega^L_{\eta}$};
      \draw (6,2.5)  node[right]{$\Omega^R_{\eta}$};
      \draw (2.5,2.5)  node[]{$\Omega^M_{\eta}$};
      \draw (0,-1)  node[below]{$\Gamma^{L,M}_{\eta}$};
      \draw (5,-1)  node[below]{$\Gamma^{M,R}_{\eta}$};
    \end{tikzpicture}
    \caption{The geometry of the asymptotic problem when $\epsilon\to 0$}
    \label{fig:geom2}
  \end{center}
\end{figure}

\begin{theorem}\label{th:convergence_result}
   As $\epsilon\to 0$, $v_{\eta,\epsilon}$ converges locally uniformly  to $v_\eta$ the unique bounded viscosity solution of 
  \begin{eqnarray}
    \label{def:HJeffective1}
    \lambda v(z)+ H^L( D v(z))  = 0 & &\hbox{if } z\in \Omega^L_{\eta},\\
    \label{def:HJeffective2}
    \lambda v(z)+ H^M( D v(z))  = 0  & &\hbox{if } z\in \Omega^M_{\eta},\\ 
    \label{def:HJeffective3}
    \lambda v(z)+ H^R( D v(z))  = 0 & &\hbox{if } z\in \Omega^R_{\eta},\\  
    \label{def:HJeffective4}
    \lambda v(z)+\max\left(E^{L,M}(\partial_{z_2}v(z)), H^{L,M}( D v^L(z),D v^M(z)) \right) = 0  & &\hbox{if } z\in \Gamma^{L,M}_\eta,\\ 
    \label{def:HJeffective5}
    \lambda v(z)+\max\left(E^{M,R}(\partial_{z_2}v(z)), H^{M,R}( D v^M(z),D v^R(z)) \right) = 0  & &\hbox{if } z\in \Gamma^{M,R}_\eta,
  \end{eqnarray}
where $v^L$ (respectively $v^M$,$v^R$) stands for $v|_{\overline{\Omega^L_{\eta}}}$ 
(respectively $v|_{\overline{\Omega^M_{\eta}}}$, $v|_{\overline{\Omega^R_{\eta}}}$),
and  that we note for short
\begin{equation}\label{def:HJeffective_short}
  \lambda v(z)+{\cal{H_\eta}}(z, Dv(z))=0.
\end{equation}
The  Hamiltonians $H^L$ and $H^R$ are defined in (\ref{eq:7}). The effective Hamiltonian $H^M$ will be defined in \S~\ref{sec:effect-hamilt-omeg}  below (note that $p\mapsto H^M (p)$ is convex).
In (\ref{def:HJeffective4}), \begin{equation}
    \label{eq:28}
 H^{L,M}(p^L,p^M) = \max \{ \;H^{+,1,L}(p^L),H^{-,1,M} (p^M)\},
\end{equation}
for   $p^L\in \R^2 $ and $p^M \in \R^2$ (in what follows, $p^L -p^M$ is colinear to $e_1$) and, for $i=L,M,R$, $p\mapsto H^{+,1,i}(p)$ (respectively $p\mapsto H^{-,1,i}(p)$) is the nondecreasing (respectively  nonincreasing) part of the Hamiltonian $H^i$ with respect to the first coordinate $p_1$ of $p$.
In  (\ref{def:HJeffective5}),  
\begin{equation}
\label{eq:31}
 H^{M,R}( p^M,p^R)  = \max \{ \;H^{+,1,M}(p^M),H^{-,1,R} (p^R)\},
\end{equation}
for  $p^M\in \R^2 $ and $p^R \in \R^2$.
\\
The effective flux limiters $E^{L,M}$ and $E^{M,R}$ will be defined in \S~\ref{sec:ergod-const-state}.
\end{theorem}
Let us list the  notions which are needed by Theorem~\ref{th:convergence_result}  and give a few comments:
\begin{enumerate}
\item Problem~(\ref{def:HJeffective_short}) is a transmission problem across the interfaces  $\Gamma^{L,M}_\eta$ and  $\Gamma^{M,R}_\eta$,
 with the respective effective transmission conditions~(\ref{def:HJeffective4}) and (\ref{def:HJeffective5}).
 The notion of viscosity solutions of (\ref{def:HJeffective_short}) 
is similar to the one defined for problem (\ref{HJaepsilon}). 
\item Note that the Hamilton-Jacobi equations in $\Omega^{L}_{\eta}$ and $\Omega^{R}_{\eta}$ are  directly inherited from (\ref{eq:5bis}): this is quite natural, 
since the Hamilton-Jacobi equation at $x\in  \Omega^{L}_{\eta}$ and $x\in  \Omega^{L}_{\eta}$ does not depend on $\epsilon$ if $\epsilon$ is small enough.
\item The effective Hamiltonian  $H^M(p)$ arising in $\Omega^M_\eta$ will be found by solving  classical  one dimensional
 cell-problems in the fast vertical variable $y_2=x_2/\epsilon$. The cell problems are one dimensional,
 since  %for any $x_1$ such that $|x_1|<\eta$ and 
for  any interval $I$
such that $ I \subset\subset (-\eta, \eta)$,   
 $  \Gamma_{\eta,\epsilon} \cap (I\times \R)$  is made of straight horizontal lines as soon as $\epsilon$ is small enough.
\item  The Hamiltonian  $H^{L,M}$  appearing in the effective transmission condition at the interface $\Gamma^{L,M}_\eta$   is built by considering
 only the effective dynamics related to $ \Omega^{i,\eta}$ which point from $\Gamma^{L,M}_\eta$
 toward $\Omega^{i,\eta}$, for $i=L,M$. The same remark holds for $H^{M,R}$ mutatis mutandis.

\item The effective flux limiters $E^{L,M}$ and $E^{M,R}$ are the only ingredients in the effective problem that  keep track of the function $g$.
 They are  constructed in \S~\ref{sec:ergod-const-state} and \ref{sec:passage-limit-as} below, see (\ref{eq:def_E}),  as 
the limit of a sequence of ergodic constants related to larger and larger domains bounded in the horizontal direction. This is reminiscent of a construction first performed in \cite{MR3299352} for 
 singularly perturbed problems in optimal control leading to Hamilton-Jacobi equations posed on a network. Later, similar constructions were used in \cite{MR3441209,MR3565416}.
\item For proving  Theorem~\ref{th:convergence_result}, 
 the chosen strategy is reminiscent of  \cite{MR3441209},
 because it relies on the construction of a single corrector,
 whereas the method proposed in  \cite{MR3299352} requires the construction of an infinite family of correctors.  This will be done in \S~\ref{sec:proof-theor-refth:c} and the slopes at infinity of the correctors will be studied in \S~\ref{sec:asympt-valu-slop}. 
\end{enumerate}

\subsection{The effective Hamiltonian  $H^M$}\label{sec:effect-hamilt-omeg}
The first step in understanding the asymptotic behavior of the value function $v_{\eta,\epsilon}$ as $\epsilon\to 0$ is to look at what happens
in $\Omega^M$, i.e. in the region where $|x_1| <\eta$.  For that, it is possible to rely on existing results, see \cite{LPV} for the first work on the topic. 
In $\Omega^M$, if the sequence of value functions  $v_{\eta,\epsilon}$ converges to $v_\eta$ uniformly as $\epsilon\to 0$, then
$v_\eta$ is a viscosity solution of  a  first order partial differential equation involving an effective Hamiltonian noted $H^M$ in (\ref{def:HJeffective2}) and in the rest of the paper. 
The latter will be obtained by solving a one-dimensional periodic boundary value problem in the fast variable  $y\in \R$, usually 
named {\sl a cell problem}. Before stating the result, it is convenient to introduce 
 the open sets $Y_\eta ^L =   \left(\eta a,  \eta b \right)+ \eta \Z$ and $Y_\eta ^R= \R \setminus \overline {Y_\eta ^L}$, and the discrete sets 
$\gamma_\eta^{a}=  \{\eta a\}+ \eta \Z$, $\gamma_\eta^{b}=  \{\eta b\}+ \eta \Z$.
 For $p\in \R^2$, $i=L, R$, we also define the Hamiltonians:
\begin{eqnarray}
  H^{-,2,i}(p) &=& \max_{\alpha \in A^i,  f^i(\alpha ) \cdot e_2 \ge 0 } (-p \cdot f^i(\alpha ) -\ell^i(\alpha )),\\
  H^{+,2,i}(p) &=& \max_{\alpha \in A^i,  f^i(\alpha ) \cdot e_2 \le 0 } (-p \cdot f^i(\alpha ) -\ell^i(\alpha )).
\end{eqnarray}
Note that $H^{+,2,i}(p)$ (respectively $H^{-,2,i}(p)$) is the nondecreasing (respectively  nonincreasing)
 part of the Hamiltonian $p\mapsto H^i(p)$ with respect to the second coordinate $p_2$ of $p$.
\begin{proposition}
  \label{sec:effect-hamilt-omeg-1}
 For any  $p\in \R^2$   there exists a  unique real number $H^M(p)$ such that the  following one dimensional
  cell-problem has a Lipschitz continuous viscosity  solution $\zeta(p,\cdot)$:
 \begin{eqnarray}
   H^R\left (  p+  \frac {d\zeta}{dy}(y) e_2 \right) = H^M(p), \quad   \hbox{  if } y\in Y_\eta ^R ,\\
   H^L\left ( p+  \frac {d\zeta}{dy}(y) e_2 \right) = H^M(p), \quad   \hbox{  if } y \in Y_\eta ^L,\\
 \max\left(   H^{+,2,R} \left (  p+  \frac {d\zeta}{dy} (y ^-) e_2 \right),   H^{-,2,L}\left (  p+  \frac {d\zeta}{dy}(y ^+) e_2 \right) \right)= H^M(p),  
  \hbox{  if } y\in \gamma_\eta^{a} ,\\
 \max\left(   H^{+,2,L}\left ( p+  \frac {d\zeta}{dy}  (y^-) e_2 \right),   H^{-,2,R}\left (  p+  \frac {d\zeta}{dy}  (y ^+) e_2 \right) \right)= H^M(p),
  \hbox{  if }   y\in \gamma_\eta^{b},\\
 \zeta \hbox{ is periodic in }y \hbox{ with period }\eta.
 \end{eqnarray}
\end{proposition}
The following lemma contains information on $H^M$: we skip its proof because it is very much like the proof of  \cite[Lemma 4.16]{MR3565416}.
\begin{lemma}
\label{sec:effect-hamilt-hm}
The function $p\mapsto H^M(p)$ is convex.
 There exists a constant $C$  
%and a modulus of continuity $\omega$ 
such that for any 
%$x,x'\in \R,\; 
$ p, p'\in \R^2$,
 \begin{eqnarray}
   \label{lem:E_lipschitz_wrt_p_2}
\mid H^M(p)-H^M (p') \mid \le C |p-p'|,
\\
% \label{lem:E_lipschitz_wrt_z_2}
% \mid    H^M(x,p)-H^M (x',p)  \mid \le C(1+|p|) |x-x'|+\omega(|x-x'|),
% \\
 \label{lem:E_coercive}
 \delta_0|p|-C\le H^M (p)\le C |p|+C.
\end{eqnarray}
\end{lemma}
 As in \cite{MR3299352,MR3441209}, we introduce three functions $ E_0^i:  \R \to \R$, $i=L,M,R$, and two functions $ E^{L,M}_0: \R \to \R$    and  $ E^{M,R}_0 : \R \to \R$:
 \begin{eqnarray}
  \label{def:E_0^i}
  E_0^i(p_2)&=&\min\left\lbrace H^i(p_2e_2+qe_1), \quad q\in \R \right\rbrace,
  \\
   \label{def:ELM_0}
  E^{L,M}_0(p_2)&=&\max\left\lbrace  E_0^L(p_2), E_0^M(p_2)\right\rbrace, 
\\
   \label{def:EMR_0}
  E^{M,R}_0(p_2)&=&\max\left\lbrace  E_0^M(p_2), E_0^R(p_2)\right\rbrace.
 \end{eqnarray}
For $i=L,M,R$, $H^{+,1,i}(p)$ (respectively $H^{-,1,i}(p)$) is the nondecreasing (respectively  nonincreasing)
 part of the Hamiltonian $p\mapsto H^i(p)$ with respect to the first coordinate $p_1$ of $p$.
For $ p_2\in \R$,  there exists a unique pair of real numbers $p^{-,i}_{1,0}(p_2)\le  p^{+,i}_{1,0}(p_2) $ such that
   \begin{eqnarray*}
 H^{-,1,i}(p_2 e_2+p_1 e_1) &=& \left\lbrace
 \begin{array}{ll}
 H^i(p_2 e_2+p_1e_1) & \mbox{ if } p_1\le p^{-,i}_{1,0}(p_2),\\
 E^i_0(p_2)  & \mbox{ if } p_1> p^{-,i}_{1,0}(p_2),
 \end{array}
 \right.    \\
 H^{+,1,i}(p_2 e_2+p_1e_1) &=& \left\lbrace
 \begin{array}{ll}
 E^i_0(p_2)  & \mbox{ if } p_1\le p^{+,i}_{1,0}(p_2),\\
 H^i(p_2 e_2+p_1e_1) & \mbox{ if } p_1> p^{+,i}_{1,0}(p_2).
 \end{array}
 \right.
  \end{eqnarray*}

\subsection{Truncated cell problems for the construction of the flux limiters $E^{M,R}$ and $E^{L,M}$}
\label{sec:ergod-const-state}
In what follows, we focus on the construction of  $E^{M,R}$ and on its properties, 
the construction of $E^{L,M}$ being completely symmetric.

\subsubsection{Zooming near the line $\Gamma^{M,R}_\eta$}\label{sec:state-constr-probl}
Asymptotically when $\epsilon\to 0$, the two lines  $\Gamma^{L,M}_\eta$ and $\Gamma^{M,R}_\eta$  appear very far from each other at the scale $\epsilon$.
 This is why we are going to introduce another geometry obtained by first zooming near 
$\Gamma^{M,R}_\eta$ at a scale $1/\epsilon$, then letting $\epsilon $ tend to $0$. 
\\
Let $\widetilde G$ be the multivalued  step function, periodic with period $1$,  such that
\begin{enumerate}
\item $\widetilde G(a)=\widetilde G(b)=[-\infty,0]$
\item $\widetilde G(t)=\{0\}$ if $t\in (a,b )$
\item $\widetilde G(t)=\{-\infty\}$ if $t \in [0,a )\cup(b,1]$
\end{enumerate}
Consider the  curve $\widetilde \Gamma_{\eta}$ defined as the graph of the multivalued function
 $\tilde g_{\eta}: x_2\mapsto \eta \widetilde G(\frac {x_2}{ \eta})+ \eta  g(\frac {x_2}{  \eta})$. 
We also define the domain $\widetilde \Omega_{\eta}^R $  (resp. $\widetilde \Omega_{\eta}^L $) as the epigraph (resp. hypograph) 
of  $\tilde g_{\eta}$:
\begin{eqnarray*}
  \widetilde \Omega_{\eta}^R= & \{ x\in \R^2 :   x_1> \tilde g_{\eta} (x_2)\},\\
  \widetilde \Omega_{\eta}^L= & \{ x\in \R^2 :   x_1< \tilde g_{\eta} (x_2)\}.
\end{eqnarray*}
The unit normal vector $\tilde n_{\eta}(x)$ at $ x\in \widetilde \Gamma_{\eta}$ is defined  as follows: setting $y_2= \frac { x_2}{\eta}$,
\begin{displaymath} 
\tilde n_{\eta}(x)= \left\{
  \begin{array}[c]{cl}
    \ds  
\left (1 + \left( g'(y_2)\right)^2 \right) ^{-1/2}  \left(  e_1 -   g' (y_2)   e_2\right)
  \quad &\hbox{ if    }\quad y_2 \notin  \Sp\\ 
\ds - e_2     \quad&\hbox{ if    }\quad y_2 = a \mod{1}\\

\ds  e_2     \quad&\hbox{ if    } \quad y_2 = b \mod{1}.
  \end{array}
\right.
\end{displaymath}
Note that $\tilde n_{\eta}(x)$ is  oriented from $\widetilde \Omega^L_{\eta}$ to $\widetilde \Omega^R_{\eta}$.

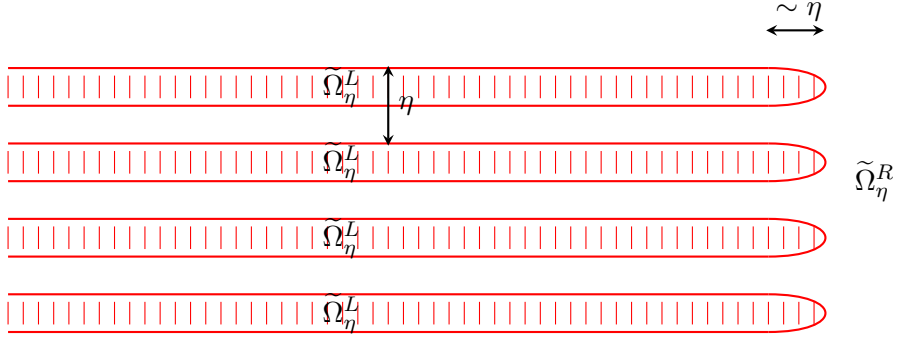
\begin{figure}[H]
  \begin{center}
    \begin{tikzpicture}[scale=1, trans/.style={thick,<->,shorten >=2pt,shorten <=2pt,>=stealth} ]
      \draw[red,thick] (0,10) -- (10,10);
      \draw[red,thick] (10,10) .. controls (11,10) and (11,9.5) .. (10,9.5);
      \draw[red,thick] (10,9.5) -- (0,9.5);
      \foreach \x in {0,...,53}
      \draw[red] (\x/5,9.6)--(\x/5,9.9);
      \draw (4,9.75)  node[right]{$\widetilde \Omega^L_{\eta}$};

      \draw[red,thick] (0,9) -- (10,9);
      \draw[red,thick] (10,9) .. controls (11,9) and (11,8.5) .. (10,8.5);
      \draw[red,thick] (10,8.5) -- (0,8.5);
      \foreach \x in {0,...,53}
      \draw[red] (\x/5,8.6)--(\x/5,8.9);
      \draw (4,8.75)  node[right]{$\widetilde \Omega^L_{\eta}$};

      \draw[red,thick] (0,8) -- (10,8);
      \draw[red,thick] (10,8) .. controls (11,8) and (11,7.5) .. (10,7.5);
      \draw[red,thick] (10,7.5) -- (0,7.5);
      \foreach \x in {0,...,53}
      \draw[red] (\x/5,7.6)--(\x/5,7.9);
      \draw (4,7.75)  node[right]{$\widetilde \Omega^L_{\eta}$};

      \draw[red,thick] (0,7) -- (10,7);
      \draw[red,thick] (10,7) .. controls (11,7) and (11,6.5) .. (10,6.5);
      \draw[red,thick] (10,6.5) -- (0,6.5);
      \foreach \x in {0,...,53}
      \draw[red] (\x/5,6.6)--(\x/5,6.9);
      \draw (4,6.75)  node[right]{$\widetilde \Omega^L_{\eta}$};

      \draw[trans] (9.9,10.5) -- (10.8,10.5) ;
      \draw (10.4,10.5)  node[above]{$\sim \eta$};

      \draw[trans] (5,10.1) -- (5,8.9) ;
      \draw (5,9.5)  node[right]{$\eta$};

%      \draw (-1,6.5)  node[left]{$\Omega^L_{\eta,\epsilon}$};
      \draw (11,8.5)  node[right]{$\widetilde \Omega^R_{\eta}$};
    \end{tikzpicture}
    \caption{The  interface $\widetilde \Gamma_{\eta}$ separates the disconnected domain $\widetilde \Omega^L_{\eta}$ and the connected domain $\widetilde \Omega^R_{\eta}$.}
    \label{fig:geom3}
  \end{center}
\end{figure}

\subsubsection{State-constrained problem in truncated domains}\label{sec:state-constr-probl-1}
 We introduce the following Hamiltonians: 
\begin{eqnarray}
\label{eq:1}
H_{ \widetilde \Gamma_{\eta}}^{-,i} (p,y)= \max_{a\in A^i \hbox{ s.t. }   f^i(a)\cdot \tilde n_{\eta}(y)\ge 0} (-p\cdot f^i(a) -\ell^i(a)), \quad \forall y\in \widetilde \Gamma_{\eta}, \forall p\in \R^2,
\\
\label{eq:34}
H_{ \widetilde \Gamma_{\eta}}^{+,i} (p,y)= \max_{a\in A^i \hbox{ s.t. }   f^i(a)\cdot \tilde n_{\eta}(y)\le 0} (-p\cdot f^i(a) -\ell^i(a)), \quad \forall y\in \widetilde \Gamma_{\eta}, \forall p\in \R^2,
\end{eqnarray}
with  $\tilde n_{\eta}(y)$ defined in \S~\ref{sec:state-constr-probl}, and, for $p^L, p^R \in \R^2 $
\begin{equation}
  \label{eq:3}
H_{ \widetilde \Gamma_{\eta}} (p^L,p^R,y)= \max \{ \;H_{ \widetilde \Gamma_{\eta}}^{+,L}(p^L,y),H_{ \widetilde \Gamma_{\eta}}^{-,R} (p^R,y)\}.
\end{equation}
In what follows, $p^L -p^R$ will always be colinear to $\tilde n_{\eta}(y)$.
 For $\rho>0$, let us set $    Y^\rho=\{y:  |y_1| < \rho \}$. For $\rho$ large enough such that $\tilde \Gamma_\eta$ is
strictly contained in  $  \{y: y_1<\rho\}$,   consider the {\sl truncated cell problem}
\begin{equation}
\label{trunc-cellp}
\left\{
    \begin{array}[c]{lll}
    H^L(Du(y)+p_2e_2)&\le \lambda_\rho( p_2)&\hbox{ if } y\in   \widetilde \Omega_{\eta}^L \cap  Y^\rho   ,  \\
    H^L(Du(y)+p_2e_2)&\ge \lambda_\rho( p_2)&\hbox{ if } y\in   \widetilde \Omega_{\eta}^L \cap  \overline{Y^\rho}   ,  \\
    H^R(Du(y)+p_2e_2)&\le \lambda_\rho( p_2)&\hbox{ if } y\in   \widetilde \Omega_{\eta}^R \cap  Y^\rho   ,  \\
    H^R(Du(y)+p_2e_2)&\ge \lambda_\rho( p_2)&\hbox{ if } y\in   \widetilde \Omega_{\eta}^R \cap  \overline{Y^\rho}   ,  \\
    H_{ \widetilde \Gamma_{\eta}} (Du^L(y)+p_2e_2 ,   Du^R(y)+p_2e_2,y)&\le \lambda_\rho( p_2)&\hbox{ if } y\in   \widetilde \Gamma_{\eta} \cap  Y^\rho   ,  \\
    H_{ \widetilde \Gamma_{\eta}} (Du^L(y)+p_2e_2 ,   Du^R(y)+p_2e_2,y)&\ge \lambda_\rho( p_2)&\hbox{ if } y\in   \widetilde \Gamma_{\eta} \cap  \overline{Y^\rho} , \\
    u \hbox{ is 1-periodic w.r.t. } y_2/\eta,
    \end{array}
\right.
\end{equation}
where the inequations are understood in the sense of viscosity. 
\begin{lemma}
\label{lem:existence_solution_truncated_cell_pb}
There is a unique  $\lambda_\rho( p_2)\in \R$ such that \eqref {trunc-cellp} admits a viscosity solution. 
For this choice of  $\lambda_\rho( p_2)$, there exists a  solution  $\chi_\rho(p_2,\cdot)$
 which is Lipschitz continuous with  Lipschitz constant
 $L$ depending on $p_2$ only (independent of $\rho$).
\end{lemma}
\begin{proof}
  We skip the proof of this lemma, since it is very much like that of \cite[ Lemma 4.6]{MR3565416}. 
\end{proof}
\subsection{ The effective flux limiter $E^{M,R}(p_2)$ and the global cell problem} 
\label{sec:passage-limit-as}
As in \cite{MR3299352,MR3565416},  using the optimal control interpretation of (\ref{trunc-cellp}), it is easy to prove that for 
a positive $K$ which may depend on  $p_2$ but not on $\rho$, and for 
all $0<\rho_1\le \rho_2$, 
\[\lambda_{\rho_1}( p_2)\leq \lambda_{\rho_2}( p_2)\leq K.\]
For $p_2\in \R$, the effective tangential Hamiltonian $E^{M, R}(p_2)$ is defined by 
\begin{equation}
  \label{eq:def_E}
  E^{M, R}(p_2)=\lim_{\rho\rightarrow \infty} \lambda_{\rho}( p_2).
\end{equation}
For  a fixed  $p_2\in\R$, the {\sl global cell-problem} reads
\begin{equation}
\label{cellpE}
\left\{
    \begin{array}[c]{lll}
    H^L(Du(y)+p_2e_2)&= E^{M, R}( p_2) &\hbox{ if } y\in   \widetilde \Omega_{\eta}^L   ,  \\
    H^R(Du(y)+p_2e_2)&= E^{M, R}( p_2) &\hbox{ if } y\in   \widetilde \Omega_{\eta}^R    ,  \\
    H_{ \widetilde \Gamma_{\eta}} (Du^L(y)+p_2e_2 ,   Du^R(y)+p_2e_2,y)&= E^{M, R}( p_2) & \hbox{ if } y\in   \widetilde \Gamma_{\eta}  ,  \\
    u \hbox{ is 1-periodic w.r.t. } y_2/\eta.
    \end{array}
\right.
\end{equation}
The following theorem is proved exactly as Theorem 4.8 in \cite{MR3565416}.
\begin{theorem}
\label{thm:stability_from_truncated_cell_pb_to_global_cell_pb}
Let $\chi_\rho( p_2,\cdot)$ be a sequence of uniformly Lipschitz continuous solutions of the truncated cell-problem \eqref{trunc-cellp} which converges to $ \chi( p_2,\cdot)$
locally uniformly in $\R^2$. Then  $ \chi(p_2,\cdot)$ is a Lipschitz continuous viscosity solution of the global cell-problem \eqref{cellpE}. 
 By subtracting  $\chi(p_2, 0)$ to $\chi_\rho(p_2,\cdot)$ and $\chi(p_2,\cdot)$, we may also assume that $\chi(p_2,0)=0$.
\end{theorem}

 \subsubsection{Comparison between  $E^{M, R}_0(p_2)$ and $E^{M, R}(p_2)$ respectively defined in (\ref{def:EMR_0}) and (\ref{eq:def_E})}\label{sec:comp-betw-e_0-1}
For $\epsilon>0$,  let us set $W_\epsilon(p_2,y)=\epsilon\chi( p_2,\frac {y-\eta e_1} {\epsilon})$. The following result is reminiscent of \cite[Theorem 4.6,iii]{MR3441209}:
 \begin{lemma}
\label{lem:rescaling_omega} 
 For any $p_2\in \R$,   
there exists a sequence    $\epsilon_n$ of positive numbers tending to $0$ as $n\to +\infty$  such that  $W_{\epsilon_n}  (p_2,\cdot)$ converges locally uniformly to a 
Lipschitz  function $y\mapsto W( p_2, y)$ (the Lipschitz does not depend on $\eta$). %\eqref{eq:v_R^rho_L_lipschitz}. 
This function is constant with respect to $y_2$ and satisfies $W( p_2,\eta e_1)=0$. It is  a viscosity solution of 
\begin{equation}
\label{W}
\begin{array}[c]{rcll}
H^R(Du(y)+p_2e_2)&=&E^{M, R}(p_2),\quad&\hbox{ if}\quad  y_1>\eta,\\
H^M(Du(y)+p_2e_2)&=&E^{M, R}(p_2),\quad&\hbox{ if}\quad  y_1<\eta.
\end{array}
\end{equation}
\end{lemma}
\begin{proof}
  It is clear that $y\mapsto W_\epsilon( p_2,y)$ is a Lipschitz continuous function with a constant $\Lambda$ independent of $\epsilon$ and that $W_\epsilon( p_2,\eta e_1)=0$.
Thus, from Ascoli-Arzela's Theorem, we may assume that  $y\mapsto W_\epsilon( p_2,y)$  converges locally uniformly to some function $y \mapsto W(p_2,y)$,  maybe after the extraction of a subsequence.
 The function $y \mapsto W(p_2,y)$ is Lipschitz continuous  with constant $\Lambda$ and  $W(p_2,\eta e_1)=0$.
 Moreover, since $W_\epsilon( p_2,y)$ is periodic with respect to $y_2$  with period $\epsilon$, $W(p_2,y)$ does not depend on $y_2$.\\
 To prove that $W(p_2,\cdot)$ is a viscosity solution of \eqref{W},
we focus on the more difficult case when $y_1<\eta$; we also restrict ourselves to  proving that  $W(p_2,\cdot)$ is a viscosity subsolution of \eqref{W}, because the proof that $W(p_2,\cdot)$ 
is a viscosity supersolution follows the same lines.\\
 Consider  $\bar y\in \R^2 $ such that $\bar y_1<\eta$, $\phi \in \cC^1(\R^2)$ and $r_0<0$  such that  $ B(\bar y,r_0)$ is contained in $\{y_1<\eta\}$  and that 
\begin{equation*}
\label{proof:rescaling1}
W(p_2,y)-\phi(y)<W(p_2,\bar y)-\phi(\bar y)=0 \mbox{ for  } y\in B(\bar y,r_0)\setminus\{\bar y\}.
\end{equation*}
We first observe   that $y\mapsto W_\epsilon( p_2,y)$ is a viscosity solution of
\begin{equation}
  \label{eq:9}
  \begin{array}[c]{rcll}
H^i(Du(y)+p_2e_2)&=&  E^{M,R}( p_2), \quad  &    \hbox{ if }y\in \Omega^i_{\eta,\epsilon} \cap B(\bar y,r_0) , \,i=L,R,\\    
H_{ \Gamma_{\eta,\epsilon}} (y,Du^L(y)+p_2e_2, Du^R(y)+p_2e_2 )&=&  E^{M,R}( p_2), \quad  &    \hbox{ if }y\in\Gamma_{\eta,\epsilon} \cap B(\bar y,r_0).
  \end{array}
\end{equation}
%We may assume that $\phi$ depends only on $y_1$ near $y=\bar y$. 
We wish to prove that
$H^M(D\phi(\bar y)+p_2e_2)\le E^{M,R}( p_2)$.
Let us argue by contradiction and assume that there exists $\theta>0$ such that
\begin{equation}
\label{proof:rescaling3}
H^M(D\phi(\bar y)+p_2e_2)= E^{M,R}( p_2)+\theta.
\end{equation}
Take $\phi_\epsilon(y)=\phi(y)+\epsilon   \zeta( D\phi(\bar y)+p_2e_2, \frac{y_2}{\epsilon})-\delta$, where
$\zeta$ is a one-dimensional periodic corrector constructed in Proposition \ref{sec:effect-hamilt-omeg-1} and $\delta>0$ is a fixed positive number.
We claim that  for $r>0$ small enough, $\phi_\epsilon$ is a viscosity supersolution of
\begin{equation}
\label{eq:10}
  \begin{array}[c]{rcll}
    H^i(Du(y)+p_2e_2)&\ge &  E^{M,R}( p_2)+\frac \theta 2 , \quad  &    \hbox{ if }y\in \Omega^i_{\eta,\epsilon} \cap B(\bar y,r) ,\\    
    H_{ \Gamma_{\eta,\epsilon}} (y,Du^L(y)+p_2e_2, Du^R(y)+p_2e_2 )&\ge&  E^{M,R}( p_2) +\frac \theta 2, \quad  &    \hbox{ if }y\in\Gamma_{\eta,\epsilon} \cap B(\bar y,r).
  \end{array}
\end{equation}
This comes from (\ref{proof:rescaling3}), the definition of $ \zeta( D\phi(\bar y)+p_2e_2, \frac{y_2}{\epsilon})$, the $\cC^1$ regularity of  $\phi$
 and the Lipschitz continuity of $H^ i$ and $ H_{ \Gamma_{\eta,\epsilon}}$ with respect to the $p$ variables.\\
Hence,  $W_\epsilon(p_2,\cdot)$ is a subsolution of (\ref{eq:9}) and  $\phi_\epsilon$ is a supersolution of (\ref{eq:10})
 in $B(\bar y,r)$.
  Moreover for $r>0$ small enough,
$ \max_{y\in \partial B(\bar y,r)}\left(W(p_2,y)-\phi(y) \right)<0$. Hence, for $\delta>0$ and $\epsilon>0$ small enough
$\max_{y\in \partial B(\bar y,r)}\left(W_\epsilon(p_2,y)-\phi_\epsilon(y) \right) \le 0$.\\
Thanks to a standard comparison principle (which holds thanks to the fact that $\frac{\theta}{2}>0$)
\begin{equation}
\label{proof:rescaling6}
\max_{y\in B(\bar y,r)}\left(W_\epsilon(p_2,y)-\phi_\epsilon(y) \right) \le 0.
\end{equation}
 Letting $\epsilon \to 0$ in \eqref{proof:rescaling6}, we deduce that $W(p_2,\bar y)\le\phi(\bar y)-\delta$,
which is in contradiction with the assumptions. 
\end{proof}
 Using Lemma~\ref{lem:rescaling_omega}, it is possible to  compare
 $E^{M, R}_0(p_2)$ and $E^{M, R}(p_2)$ respectively defined in (\ref{def:EMR_0}) and (\ref{eq:def_E})
 \begin{proposition}
\label{cor:E_bigger_than_E_0}
For any $p_2 \in \R$,
\begin{equation}
\label{eq:cor:E_bigger_than_E_0}
E^{M, R}(p_2) \geq  E^{M, R}_0(p_2).
\end{equation}
\end{proposition}
\begin{proof}
Thanks to Lemma \ref{lem:rescaling_omega}, the function $y\mapsto W(p_2,y)$ is a viscosity solution
$H^M(Du(y)+p_2e_2)=E^{M, R}(p_2)$  in $\{y: y_1<\eta\}$.  Keeping in mind that  $W(p_2,y)$ is independent of $y_2$,
 we see that  for almost all $y_1<\eta$, $E^{M, R}(p_2)= H^M(\partial_{y_1}W(p_2,y_1)e_1+p_2e_2) \ge E_0^M(p_2)$, from (\ref{def:E_0^i}).
Similarly, we show that  $E^{M, R}(p_2)= H^R(\partial_{y_1}W(p_2,y_1)e_1+p_2e_2) \ge E_0^R(p_2)$ at almost $y_1>\eta$, and we conclude using (\ref{def:EMR_0}).
\end{proof}
\subsubsection{Asymptotic values of the slopes of $\chi$ as $y_1\to \infty$}\label{sec:asympt-valu-slop}
From Proposition \ref{cor:E_bigger_than_E_0} and the coercivity of the Hamiltonians $H^i$, $i=M,R$,
 the following numbers are well defined for all $p_2 \in \R$:
\begin{eqnarray}
\label{eq:5}
\overline{\Pi}^M(p_2)\!= \!\min\left\lbrace q\in \R : H^M( p_2e_2+qe_1)=H^{-,1,M}( p_2e_2+qe_1)=E^{M,R}(p_2) \right\rbrace\\
\label{eq:6}
\widehat{\Pi}^M(p_2)\!= \!\max\left\lbrace q\in \R : H^M( p_2e_2+qe_1)=H^{-,1,M}( p_2e_2+qe_1)=E^{M,R}(p_2) \right\rbrace\\
\label{eq:23}
\overline{\Pi}^R(p_2)\!= \!\min\left\lbrace q\in \R : H^R( p_2e_2+qe_1)=H^{+,1,R}( p_2e_2+qe_1)=E^{M,R}(p_2) \right\rbrace\\
\label{eq:24}
\widehat{\Pi}^R(p_2)\!= \!\max\left\lbrace q\in \R : H^R( p_2e_2+qe_1)=H^{+,1,R}( p_2e_2+qe_1)=E^{M,R}(p_2) \right\rbrace
\end{eqnarray}
\begin{remark}
\label{rmk:equality_bar_Pi_and_hat_Pi}
From the  convexity  of the Hamiltonians $H^i$ and $H^{\pm,1,i}$, we deduce that if for $i=M,R$, $E^i_0(p_2)<E ^{M,R}(p_2)$,  then 
$\overline{\Pi}^i(p_2)=\widehat{\Pi}^i(p_2)$. In this case, we will use the notation
\begin{equation}
\label{eq:special_notation_for_Pi}
\Pi^i(p_2)=\overline{\Pi}^i(p_2)=\widehat{\Pi}^i(p_2).
\end{equation}
\end{remark}
 Propositions \ref{cor:slopes_omega} and \ref{cor:control_slopes_W} below, which will be proved in Appendix~\ref{sec:proofs-prop-refc},  
provide information on the growth of $y\mapsto \chi(p_2,y)$ as $|y_1|\to \infty$, where $\chi$ is obtained in Theorem~\ref{thm:stability_from_truncated_cell_pb_to_global_cell_pb} and is a solution of the cell problem (\ref{cellpE}):
\begin{proposition}
\label{cor:slopes_omega}
With  $\Pi^i(p_2)\in \R$  defined in \eqref{eq:special_notation_for_Pi} for $i=M,R$,
\begin{enumerate}
\item If $E^{M,R}( p_2)>E_0^R(p_2)$, then,  there exist  $\rho^*=\rho^*(p_2)  >0$ and $M^*= M^*(p_2)   \in \R$
such that, for all $y\in [\rho^*,+\infty)\times \R$, $h_1\ge 0$ and $h_2\in \R$,
\begin{equation}
\label{slope32}
\chi( p_2, y+h_1e_1+h_2 e_2)-\chi( p_2, y)\geq  \Pi^R(p_2) h_1-M^*.
\end{equation}
\item  If $E^{M,R}( p_2)>E_0^M(p_2)$, then, there exist  $\rho^*=\rho^*(p_2)  >0$ and $M^*= M^*(p_2)   \in \R$
such that, for all  $y\in (-\infty,-\rho^*]\times \R$, $h_1\ge 0$ and $h_2\in \R$,
\begin{equation}
\label{slope3_bis2}
\chi( p_2, y-h_1e_1+h_2 e_2)-\chi( p_2, y)\geq - \Pi^M(p_2) h_1-M^*.
\end{equation}
\end{enumerate}
\end{proposition}

\begin{proposition}
\label{cor:control_slopes_W}
For  $p_2 \in \R$,  $y\mapsto W(p_2,y)$ defined in Lemma~\ref{lem:rescaling_omega} satisfies 
\begin{eqnarray}
\label{cor:control_slopes_W1}
\overline{\Pi}^R(p_2)\le \partial_{y_1}W(p_2,y) \le \widehat{\Pi}^R(p_2) & &\hbox{ for a.a. } y\in (\eta,+\infty)\times \R,\\
\label{cor:control_slopes_W2}
 \overline {\Pi}^M(p_2) \le \partial_{y_1}W(p_2,y) \le  \widehat{\Pi}^M(p_2) & &\hbox{ for a.a. }y\in (-\infty,\eta)\times\R,
\end{eqnarray}
and for all $y$:
\begin{equation}
\label{eq:control_slopes_W_summary}
- \widehat{\Pi}^M(p_2) (y_1-\eta)^-
+ \overline{\Pi}^R(p_2) (y_1-\eta)^+ \le W(p_2,y) \le 
- \overline{\Pi}^M(p_2) (y_1-\eta)^-
+ \widehat{\Pi}^R(p_2) (y_1-\eta)^+
.
\end{equation}
\end{proposition}

\subsection{Proof of Theorem \ref{th:convergence_result}}
\label{sec:proof-theor-refth:c}

\subsubsection{ A reduced set of test-functions}\label{sec:reduced-set-test}
From \cite{MR3621434} and \cite{imbert:hal-01073954}, 
 we may use an equivalent definition for the  viscosity  solution of (\ref{def:HJeffective_short}).
We focus on the transmission condition at the interface $\Gamma_\eta^{M,R}$, because the same kind of arguments apply to the transmission at  $\Gamma_\eta^{L,M}$. Theorem \ref{th:restriction_set_of_test_functions} below, which is reminiscent of  \cite[Theorem~2.7]{MR3621434}, will tell us that the transmission condition on  $\Gamma_\eta^{M,R}$ can be tested with a reduced set of test-functions. 
\begin{definition}
\label{def:test_functions_set_restricted}
Recall that $\overline{\Pi}^i$ and $\widehat{\Pi}^i$, $i=M,R$, have been introduced in (\ref{eq:5})-~(\ref{eq:24}).
Let $\Pi:   \Gamma_\eta^{M,R}\times \R \to \R^2$,  $(z, p_2)  \mapsto \left( \Pi^M(z,p_2),\Pi^R(z,p_2) \right)$ be such that, for all $(z,p_2)$
\begin{equation}
\label{eq:39}
  \begin{split}
\overline{\Pi}^M (p_2)   &\le   \Pi^M(z,p_2)   \le  \widehat{\Pi}^M(p_2). \\
\overline{\Pi}^R (p_2)   &\le   \Pi^R(z,p_2)   \le  \widehat{\Pi}^R(p_2) .
  \end{split}
\end{equation}
For $\bar z\in  \Gamma_\eta^{M,R}$, the reduced set of test-functions  $\cR^\Pi(\bar z)$ associated to the map $\Pi$ is
the set of the functions  $\varphi\in \cC^0 (\R^2)$ such that there exists a $\cC^1$ function $\psi: \Gamma_\eta^{M,R}   \to \R$ with 
\begin{equation}
\label{eq:40}
\varphi(z+ t e_1)= \psi(z)+   \left( \Pi^R\left(\bar z, \partial_{z_2}\psi(\bar z) \right) 1_{t>0}+\Pi^M \left(\bar z,  \partial_{z_2}\psi(\bar z)\right) 1_{t<0} \right) t.
\end{equation}
\end{definition}
The following theorem is reminiscent of \cite[Theorem~2.7]{MR3621434}.
\begin{theorem} \label{th:restriction_set_of_test_functions}
 Let $u:\R^2\to \R$ be a subsolution (resp. supersolution) of (\ref{def:HJeffective2}) and (\ref{def:HJeffective3}).
Consider a map $\Pi:\Gamma_\eta^{M,R}\times  \R\to \R^2$, $(z, p_2)  \mapsto \left( \Pi^M(z,p_2),\Pi^R(z,p_2) \right)$ such that (\ref{eq:39}) holds for all $(z,p_2)\in \Gamma_\eta^{M,R}\times \R$.
\\
We assume furthermore that $u$ is Lipschitz continuous in $\Gamma_\eta^{M,R}+ B(0,r)$ for some $r>0$.
The function $u$ is a subsolution (resp. supersolution) of (\ref{def:HJeffective5})
 if and only if for any $z\in \Gamma_\eta^{M,R}$  and for all $\varphi \in \cR^\Pi(z)$ such that $u-\varphi$ has a local maximum (resp. local minimum) at $z$,
\begin{equation}
\label{eq:th_restriction_set_of_test_functions}
\lambda u(z)+\max\left(E^{M,R}(\partial_{z_2}\varphi(z)), H^{M,R}( D \varphi^M(z),D \varphi^R(z)) \right) \le 0, \quad (\hbox{resp. }\ge 0).
\end{equation}
%where the meaning of $D \varphi^L$ and $D \varphi^R$  is given in Definition~\ref{adtest}.
\end{theorem}
\begin{proof}
The proof  follows the lines of that of \cite[Theorem~2.7]{MR3621434} and is also given in \cite[Appendix C]{MR3565416}
% and is given in Appendix \ref{appendix:proof_th_restriction_set_of_test_functions} for the reader's convenience.
%It is worth to note that Lemma \ref{sec:effect-hamilt-gamma-1} is  important in order to use the arguments 
%contained in the proof of \cite[Theorem~2.7]{MR3621434}.
\end{proof}
\begin{remark}
  \label{sec:reduced-set-test-1}
In the statement of Theorem \ref{th:restriction_set_of_test_functions}, we have chosen to restrict ourselves to functions that are Lipschitz continuous in $\Gamma^ {M,R}_\eta+ B(0,r)$ (this property makes the proof simpler); indeed, since the functions $v_{\eta, \epsilon}$ are  Lipschitz continuous with a Lipschitz constant $\Lambda$ independent of $\epsilon$ (and also of $\eta$), the relaxed semi-limits of $v_{\eta, \epsilon}$ as $\epsilon\to 0$ are also  Lipschitz continuous with the same Lipschitz constant $\Lambda$, see (\ref{eq:84}) below and  \cite[Remark 3.1]{MR3565416}. \\
In fact, a more general version of  Theorem \ref{th:restriction_set_of_test_functions} can be stated for any  lower semi-continuous supersolution, and for the  upper semi-continuous subsolutions $u$ such that  for all $z\in \Gamma^ {M,R}_\eta$, $ u(z)= \limsup_{z'\to z, z'\in \Omega^i_\eta} u(z')$, $\forall i=M,R$,
as in \cite{MR3621434,imbert:hal-01073954}.
\end{remark}
\subsubsection{Proof of Theorem~\ref{th:convergence_result}}
\label{sec:proof-theor-refth:c-1}
Let us  consider the relaxed semi-limits
 \begin{equation}
\label{eq:84}
 \overline{v_\eta}(z)={\limsup_\epsilon}^{*} {v}_{\eta,\epsilon}(z)=\limsup_{z'\to z, \epsilon\to 0}{v}_{\eta,\epsilon}(z')
 \quad \mbox{ and } \quad \underline{v_\eta}(z)=\underset{\epsilon}{{\liminf}_{*}}{v}_{\eta,\epsilon}(z)
=\liminf_{z'\to z, \epsilon\to 0}{v}_{\eta,\epsilon}(z').
 \end{equation}
Note that  $ \overline{v_\eta}$ and $ \underline{v_\eta}$ are well defined, 
since  $\left( v_{\eta,\epsilon}\right)_\epsilon$ is uniformly bounded.  We will prove that $\overline{v_\eta}$  and $\underline{v_\eta}$  are respectively a subsolution and a supersolution of \eqref{def:HJeffective_short}.
It is classical to check that the functions $\overline{v_\eta}(z)$ and $\underline{v_\eta}(z)$
 are respectively a bounded subsolution and a bounded supersolution in $\Omega^i_\eta$, $i=L,M,R$, of
\begin{equation}
\label{eq:85}
\l u(z)+H^i(Du(z))= 0.
\end{equation}
From comparison theorems proved in \cite{barles2013bellman,imbert:hal-01073954,oudet2014}, this will imply that  $\overline{v_\eta}=\underline{v_\eta}=v_\eta=\lim_{\epsilon\to 0} v_{\eta,\epsilon}$.
 We just have to check the transmission conditions (\ref{def:HJeffective4}) and (\ref{def:HJeffective5}), and it is enough to focus on the latter, since the former is dealt with in a very same manner. 
\\
We focus on    $\overline{v_\eta}$ since the proof for  $\underline{v_\eta}$ is similar.
% the transmission condition on $ \Gamma_\eta^{M,R}$ since the one on $ \Gamma_\eta^{L,M}$ is proved in the same manner.
\\
 
We are going to use Theorem \ref{th:restriction_set_of_test_functions} with the special choice for
 the map $\Pi:    \R \to \R^2$: 
$\Pi(p_2)= \left( \widehat \Pi^M   (p_2) ,\overline \Pi^R   (p_2)       \right)$. 
Note that  Theorem \ref{th:restriction_set_of_test_functions} can indeed be applied,
 because, $\overline{v_\eta}$ is Lipschitz continuous, see Remark~\ref{sec:reduced-set-test-1}.
Take $\bar z\in \Gamma_\eta^{M,R}$ and a  test-function $\varphi\in \cR^\Pi(\bar z)$, i.e. of the form 
\begin{equation}
\label{eq:86}
\varphi(z+ t e_1)= \psi(z)+   \left( \overline \Pi^R\left( \partial_{z_2}\psi(\bar z) \right) 1_{t>0}+
\widehat \Pi^M \left(  \partial_{z_2}\psi(\bar z)\right) 1_{t<0} \right) t, \quad \forall z\in  \Gamma_\eta^{M,R} , t\in \R,
\end{equation}
for  a $\cC^1$ function $\psi: \Gamma_\eta^{M,R} \to \R$,  such that $\overline{v_\eta}-\varphi$ has a strict local maximum at $\bar z$ and that $\overline{v_\eta}(\bar z)=\varphi(\bar z)$. \\
 Let us argue by contradiction with (\ref{eq:th_restriction_set_of_test_functions}) and assume that 
\begin{equation}
\label{eq:proof_convergence_sub_contradiction}
\lambda \varphi(\bar z)+ \max\left(E^{M, R}(\partial_{z_2}\varphi(\bar z)), H^{M,R}(D \varphi^M(\bar z),D \varphi^R(\bar z)) \right)=\theta >0.
\end{equation}
From  (\ref{eq:86}),  we see that $
 H^{M,R}(D \varphi^M(\bar z),D \varphi^R(\bar z))\le E^{M, R} (\partial_{z_2}\varphi(\bar z))$ and
 (\ref{eq:proof_convergence_sub_contradiction}) is equivalent to 
\begin{equation}
\label{eq:proof_convergence_sub_contradiction_bis}
\lambda \psi(\bar z)+E^{M, R}(   \partial_{z_2}\psi(\bar z) )=\theta>0.
\end{equation}
Let  $\chi( \partial_{z_2}\psi(\bar z),\cdot)$ be a solution of (\ref{cellpE})
 such that $\chi(\partial_{z_2}\psi(\bar z),0)=0$ (see Theorem~\ref{thm:stability_from_truncated_cell_pb_to_global_cell_pb}),
 and $W(\partial_{z_2}\psi(\bar z),z_1)=\lim_{\epsilon\to 0}\epsilon\chi(\partial_{z_2}\psi(\bar z),\frac {z-\eta e_1} {\epsilon})$.
\paragraph{Step 1}
Hereafter, we will consider a small positive radius $r$ such that $r< \eta/4$. Then for $\epsilon$ small enough, $\Omega_{\eta, \epsilon}^i \cap B(\bar z, r)= \left( \eta e_1 + \epsilon \widetilde \Omega_{\eta}^i \right)\cap B(\bar z, r)$ for $i=L,R$.  We claim that for $\epsilon$ and $r$ small enough, the function $\varphi^\epsilon$:
 \[\varphi^\epsilon(z)=\psi(\eta e_1+ z_2 e_2)+\epsilon\chi( \partial_{z_2}\psi(\bar z),\frac {z-\eta e_1} {\epsilon})\] 
is a viscosity supersolution of 
\begin{equation}
\label{cellpE_modifed}
\left\{
    \begin{array}[c]{rcll}
   \lambda \varphi^\epsilon(z)+ H^i(D\varphi^\epsilon(z))& \ge& \frac{\theta}{2}\quad &\hbox{ if } z\in\Omega^i_{\eta,\epsilon}\cap B(\bar z,r),\; i=L,R, \\
   \lambda \varphi^\epsilon(z)+H_{\Gamma_{\eta, \epsilon}} (z,D\left( \varphi^\epsilon\right)^L(z),D\left( \varphi^\epsilon\right)^R(z)) &\ge& \frac{\theta}{2}&\hbox{ if } z\in\Gamma_{\eta,\epsilon}\cap B(\bar z,r),
    \end{array}
\right.
\end{equation}
where $H^i$ and $H_{\Gamma_{\eta, \epsilon}} $ are defined in (\ref{eq:7})-(\ref{eq:8}).
\\
Indeed, if $\xi$ is a test-function in $\cR_{\eta,\epsilon}$ such that $\varphi^\epsilon-\xi$ has a local minimum at
  $z^\star\in B(\bar z,r)$, then, from the definition of $\varphi^\epsilon$,
 $y\mapsto \chi( \partial_{z_2}\psi(\bar z),y-\frac \eta \epsilon e_1)-\frac{1}{\epsilon}\left(\xi(\epsilon y)-\psi(\eta e_1+\epsilon y_2 e_2) \right)$
 has a local minimum at $\frac{z^\star}{\epsilon}$.
\\
If $\frac{z^\star - \eta e_1}{\epsilon}\in \widetilde \Omega^i_\eta$, for $i=L$ or  $R$, then 
$H^i(D\xi(z^\star) -\partial_{z_2} \psi(\eta e_1+ z^\star_2 e_2) e_2  + \partial_{z_2}\psi(\bar z) e_2)
\ge E^{M,R}(\partial_{z_2}\psi(\bar z))$.
From the regularity properties of   $H^i$,
 \begin{displaymath}
H^i(D\xi(z^\star) -\partial_{z_2} \psi(\eta e_1+ z^\star_2 e_2) e_2  + \partial_{z_2}\psi(\bar z) e_2)
= H^i(D\xi(z^\star)) + o_{r\to 0}(1),
\end{displaymath}
thus
\begin{equation*}
\label{eq:proof_convergence_case1_nb1}
\lambda \varphi^\epsilon(z^\star)+ H^i(D\xi(z^\star)) 
\ge E^{M,R}(\partial_{z_2}\psi(\bar z)) +\lambda
\left(\psi(\eta e_1+ z^\star_2 e_2)+\epsilon\chi( \partial_{z_2}\psi(\bar z),\frac {z^\star-\eta e_1} {\epsilon})\right)
+o_{r\to 0}(1).  
\end{equation*}
From \eqref{eq:proof_convergence_sub_contradiction_bis}, this implies that 
\begin{equation*}
\lambda \varphi^\epsilon(z^\star)+ H^i(D\xi(z^\star))\ge \theta +\lambda \epsilon\chi( \partial_{z_2}\psi(\bar z),\frac {z^\star-\eta e_1} {\epsilon})+o_{r\to 0}(1).
\end{equation*}
Recall that the function $y\mapsto \epsilon \chi( \partial_{z_2}\psi(\bar z),\frac {y-\eta e_1} {\epsilon})$
 converges locally uniformly to  $y\mapsto W(\partial_{z_2}\psi(\bar z) ,y)$,
which is a Lipschitz continuous function, independent of $y_2$ and such that   $W(\partial_{z_2}\psi(\bar z),0)=0$.
 Therefore, for $ \eta$ and $r$ small enough, $\lambda \varphi^\epsilon(z^\star)+ H^i(D\xi(z^\star))  
 \ge \frac{\theta}{2}$.
\\
If $\frac{ z^\star -\eta e_1}{\epsilon}\in \widetilde \Gamma_\eta$, then, we have 
\[H^{+,L}_{\widetilde \Gamma_\eta} (D\xi^L(z^\star) -\partial_{z_2} \psi(\eta e_1+ z^\star_2 e_2) e_2  + \partial_{z_2}\psi(\bar z) e_2, \frac{ z^\star -\eta e_1}{\epsilon}   )
\ge E^{M,R}(\partial_{z_2}\psi(\bar z))\]
or 
\[H^{-,R}_{\widetilde \Gamma_\eta} (D\xi^R(z^\star) -\partial_{z_2} \psi(\eta e_1+ z^\star_2 e_2) e_2  + \partial_{z_2}\psi(\bar z) e_2, \frac{ z^\star -\eta e_1}{\epsilon})\ge E^{M,R}(\partial_{z_2}\psi(\bar z)).\]
Since the Hamiltonians $H^{\pm,i}_{\widetilde \Gamma_\eta}$ enjoy the same regularity properties as  $H^{ i}$,
 it is possible to use the same arguments as  in the case when $\frac{z^\star-\eta e_1}{\epsilon}\in \Omega^i$.
 For  $r$ and $\epsilon$ small enough, 
\[\lambda \varphi^\epsilon(z^\star)+H_{\Gamma_{\eta, \epsilon}} (z^\star,D\left( \varphi^\epsilon\right)^L(z^\star),D\left( \varphi^\epsilon\right)^R(z^\star)) \ge \frac{\theta}{2}.
   \] 
The claim that $\varphi^\epsilon$  is a supersolution of \eqref{cellpE_modifed} is proved.

\paragraph{Step 2} Let us prove that there exist some positive constants $K_r>0$ and  $\epsilon_0>0$ such that
\begin{equation}
\label{eq:proof_convergence_case1_nb3}
 v_{\eta,\epsilon}(z)+K_r\le \varphi^\epsilon(z), \quad \forall z\in \partial B(\bar z, r), \;\forall \epsilon\in (0, \epsilon_0).
\end{equation}
Indeed, since $\overline{v_\eta}-\varphi$ has a strict local maximum at $\bar z$ 
and since $\overline{v_\eta}(\bar z)=\varphi(\bar z)$, 
there exists a positive constant $\tilde K_r>0$ such that 
$\overline{v_\eta}(z)+\tilde K_r\le \varphi(z)$ for any $z\in \partial B(\bar z, r)$.
Since $\ds \overline{v_\eta}={\limsup_\epsilon}^{*} v_{\eta,\epsilon}$,  there exists $\tilde \epsilon_0>0$ such that 
\begin{equation}
\label{eq:87}
v_{\eta,\epsilon}(z)+\frac{\tilde K_r}{2}\le \varphi(z)\quad \hbox{
for any $0<\epsilon<\tilde \epsilon_0$ and $z\in \partial B(\bar z, r)$}.
\end{equation}
On the other hand, from \eqref{eq:control_slopes_W_summary} in Proposition \ref{cor:control_slopes_W}, 
\begin{equation}
 \label{eq:88}
 \begin{split}
   &\psi(z_2 e_2+\eta e_1)+W( \partial_{z_2}\psi(\bar z),z)\\ \ge&
 \psi(z_2 e_2+\eta e_1)+   \left( \overline \Pi^R   ( \partial_{z_2}\psi(\bar z) ) 1_{z_1>\eta}
+\widehat \Pi^M( \partial_{z_2}\psi(\bar z)) 1_{z_1<\eta} \right) (z_ 1-\eta) = \varphi(z).    
 \end{split}
\end{equation}
Moreover,  $z\mapsto \varphi^\epsilon(z)$ converges locally uniformly to
 $z\mapsto \psi(\eta e_1+z_2e_2)+W( \partial_{z_2}\psi(\bar z),z)$ as $\epsilon$ tends to $0$.
By collecting the latter observation, (\ref{eq:88}) and (\ref{eq:87}), we get \eqref{eq:proof_convergence_case1_nb3} 
for some constants $K_r>0$ and $\epsilon_0>0$.

\paragraph{Step 3}
From the previous steps, we find by comparison that for $r$ and $\epsilon$ small enough, 
\begin{displaymath}
 v_{\eta,\epsilon}(z)+K_r\le \varphi^\epsilon(z) \quad \quad \forall z \in B(\bar z, r).
\end{displaymath}
Setting $z=\bar z$ and taking the $\limsup$ as  $\epsilon\to 0$,  we obtain
\begin{displaymath}
\overline{v_\eta}(\bar z)+K_r\le \psi(\bar z)=\varphi(\bar z)=\overline{v_\eta}(\bar z),
\end{displaymath}
which cannot happen. The proof is completed. \qed
\begin{remark}
For the proof of the supersolution inequality, the test-function $\varphi$ should be chosen  of the form 
\begin{displaymath}
\varphi(z+ t e_1)= \psi(z)+   \left( \widehat \Pi^R\left( \partial_{z_2}\psi(\bar z) \right) 1_{t>0}+
\overline \Pi^M \left(  \partial_{z_2}\psi(\bar z)\right) 1_{t<0} \right) t, \quad \forall z\in  \Gamma_\eta^{M,R} , t\in \R,
\end{displaymath}
where $\psi\in \cC^1(\R)$.
\end{remark}

\section{The second passage to the limit: $\eta$ tends to $0$}
\label{sec:second-passage-limit}

We now aim  at passing to the limit in (\ref{def:HJeffective_short}) as $\eta$ tends to $0$.
Recall that $\Omega^L$, $\Omega^R$ and $\Gamma$ are defined in  (\ref{eq:104}).
%We set $\Omega^L=\{x\in \R^2, x_1<0\}$, $\Omega^R=\{x\in \R^2, x_1>0\}$ and $\Gamma=\{x\in \R^2, x_1=0\}$.

\subsection{Main result}\label{sec:main-result-1}
\begin{theorem}\label{sec:main-result-2}
   As $\eta\to 0$, $v_{\eta}$ converges locally uniformly  to $v$, the unique bounded viscosity solution of 
(\ref{eq:14})-~(\ref{eq:17}), ( for short (\ref{eq:16})).
\\
In the transmission condition (\ref{eq:17}), namely
\begin{displaymath}
    \lambda v(z)+\max\left(E(\partial_{z_2}v(z)),  \;H^{+,1,L}(D v^L(z)),\;H^{-,1,R} (D v^R(z) \right) = 0  \hbox{if } z\in \Gamma,
\end{displaymath}
 the effective flux limiter $E$ is given for $p_2\in \R$ by
\begin{equation}\label{eq:18}
  E(p_2)= \max(E^{L,M}(p_2), E^{M,R}(p_2) ),
\end{equation}
where $E^{L,M}$ and $ E^{M,R}$ are defined in \S \ref{sec:passage-limit-as}, see ~(\ref{eq:cor:E_bigger_than_E_0}).
\end{theorem}

\begin{remark}
  \label{sec:main-result-6}
It is striking that, in  (\ref{eq:18}), the effective flux limiter $E(p_2)$ can be deduced explicitly from the limiters $E^{L,M}(p_2)$ and $E^{M,R}(p_2)$
 obtained in \S~\ref{sec:effect-probl-obta}.  
\end{remark}

\bigskip

Let us  consider the relaxed semi-limits
 \begin{equation}
 \label{def:v_tilde_overline_underline}
 \overline{{v}}(z)={\limsup_\eta}^{*} {v}_\eta(z)=\limsup_{z'\to z, \eta\to 0}{v}_\eta(z') \quad \mbox{ and } \quad \underline{{v}}(z)=\underset{\eta}{{\liminf}_{*}}{v}_\eta(z)=\liminf_{z'\to z, \eta\to 0}{v}_\eta(z').
 \end{equation}
Note that  $ \overline{{v}}$ and $ \underline{{v}}$ are well defined, since  $\left( v_\eta\right)_\eta$ is uniformly bounded by $M_\ell /\lambda$, see (\ref{eq:4}).
It is classical to check that the functions $\overline{v}(z)$ and $\underline{{v}}(z)$ are respectively a bounded subsolution and a bounded supersolution in $\Omega^i$ of
\begin{equation}
  \label{eq:5bisnew}
\l u(z)+H^i(Du(z))= 0.
\end{equation}
To find the effective transmission on $\Gamma$, we shall proceed as in \cite{MR3299352,MR3441209,MR3565416} and consider cell problems in larger and larger bounded domains.
\subsection{Proof of Theorem~\ref{sec:main-result-2}}
\subsubsection{State-constrained problem in truncated domains}\label{sec:state-constr-probl-2}
Let us fix $p_2\in \R$. For $\rho>1$, we consider the one dimensional {\sl truncated cell problem}:
\begin{equation}
\label{eq:20}
\left\{
    \begin{array}[c]{l}
 \ds   H^L\left(\frac{du}{dy} (y)+p_2e_2\right)\le \mu_\rho( p_2),\quad\quad\quad\hfill\hbox{ if } y\in (-\rho,-1),  \\
   \ds H^L\left(\frac{du}{dy} (y)+p_2e_2\right)\ge \mu_\rho( p_2),\quad\quad\quad\hfill\hbox{ if } y\in [-\rho,-1),  \\
   \ds H^M\left(\frac{du}{dy} (y)+p_2e_2\right)=\mu_\rho( p_2),\quad\quad\quad\hfill\hbox{ if } y\in (-1,1),  \\
    \ds H^R\left(\frac{du}{dy} (y)+p_2e_2\right)\le \mu_\rho( p_2),\quad\quad\quad\hfill\hbox{ if } y\in (1,\rho),  \\
    \ds H^R\left(\frac{du}{dy} (y)+p_2e_2\right)\ge \mu_\rho( p_2),\quad\quad\quad\hfill\hbox{ if } y\in (1,\rho],  \\
    \ds \max\left(E^{L,M}(p_2), H^{L,M}\left( \frac{d  u^L}{dy}(-1^-)+p_2e_2 ,  \frac{d u^M}{dy}(-1^+)+p_2e_2 \right) \right) =\mu_\rho( p_2),\\ %\hfill \hbox{if } y=-1, \\
    \ds \max\left(E^{M,R}(p_2), H^{M,R}\left( \frac{d u^M}{dy} (1^-)+p_2e_2 ,  \frac{du^R}{dy} (1^+)+p_2e_2 \right) \right) =\mu_\rho( p_2).
%, \\\hfill \hbox{if } y=1. \\
    \end{array}
\right.
\end{equation}
Exactly as in \cite{MR3565416}, we can prove the following lemma:
\begin{lemma}
\label{sec:state-constr-probl-3}
There is a unique  $\mu_\rho( p_2)\in \R$ such that (\ref{eq:20}) admits a bounded solution. 
For this choice of  $\mu_\rho( p_2)$, there exists a  solution  $y\mapsto \psi_\rho(p_2,y)$ which is Lipschitz continuous with  a Lipschitz constant
 $L$ depending on $p_2$ only (independent of $\rho$).
\end{lemma}
It is also possible to check that there  exists a  scalar constant $K$ such that for all real numbers $\rho_1$ and $\rho_2$ such that $\rho_1\leq \rho_2$,
\[\mu_{\rho_1}( p_2)\leq \mu_{\rho_2}( p_2)\leq K.\]
From this property, it is possible to pass to the limit  as $\rho\to +\infty$: the effective tangential Hamiltonian $E( p_2)$ is defined by 
\begin{equation}
\label{eq:21}
E( p_2)=\lim_{\rho\rightarrow \infty} \mu_{\rho}( p_2).
\end{equation}
\subsubsection{The global cell problem}\label{sec:global-cell-problem}
Fixing  $p_2\in\R$, the {\sl global cell-problem} reads
\begin{equation}
\label{eq:22}
\left\{
    \begin{array}[c]{l}
      \ds   H^L\left(\frac{du}{dy} (y) e_1+p_2e_2\right)=E( p_2),\quad\quad\quad\hfill\hbox{ if } y<-1,  \\
      \ds H^M\left(\frac{du}{dy} (y)e_1+p_2e_2\right)=E( p_2),\quad\quad\quad\hfill\hbox{ if } y\in (-1,1),  \\
      \ds H^R\left(\frac{du}{dy} (y)e_1+p_2e_2\right)=E( p_2),\quad\quad\quad\hfill\hbox{ if } y> 1,  \\
      \ds \max\left(E^{L,M}(p_2), H^{L,M}\left( \frac{d  u}{dy}(-1^-)e_1+p_2e_2 ,  \frac{d u}{dy}(-1^+)e_1+p_2e_2 \right) \right) =E( p_2),\\
% \hfill \hbox{if } y=-1, \\
      \ds \max\left(E^{M,R}(p_2), H^{M,R}\left( \frac{d u}{dy} (1^-) e_1+p_2e_2 ,  \frac{du^R}{dy} (1^+)e_1 +p_2e_2 \right) \right) =E( p_2). %\\\hfill \hbox{if } y=1. \\
    \end{array}
\right.
\end{equation}
Exactly as in \cite{MR3565416}, we obtain the existence of a solution of the global cell problem by passing to the limit in (\ref{eq:20}) as $\rho\to +\infty$:
\begin{proposition}[Existence of a global corrector]
\label{sec:state-constr-probl-5}
For  $p_2\in \R$, there exists $\psi( p_2,\cdot)$ a Lipschitz continuous viscosity solution of (\ref{eq:22})
such that $\psi( p_2,0)=0$.  For $\eta>0$,  setting $W_\eta( p_2,y)=\eta\psi( p_2,\frac y {\eta})$, 
there exists a sequence    $\eta_n$  such that  $W_{\eta_n}  ( p_2,\cdot)$ converges locally uniformly to a 
Lipschitz  function $y\mapsto W( p_2, y)$,  with the same Lipschitz constant as $\psi$.  The function $W$ is  a viscosity solution of 
\begin{equation}
\label{eq:19}
H^i\left( \frac {du}{dy_1} (y_1) e_1 +p_2e_2\right)=E( p_2)\quad\hbox{ if } y_1e_1\in \Omega^i,
\end{equation}
and satisfies $W( p_2,0)=0$. Moreover, 
\begin{equation*}
  E( p_2)\ge \max\left\lbrace  E_0^L(p_2), E_0^R(p_2)\right\rbrace.
\end{equation*}
\end{proposition}

\subsubsection{Proof of (\ref{eq:18})}
In view of Proposition \ref{sec:state-constr-probl-5}, the following numbers are well defined for all   $p_2 \in \R$:
\begin{eqnarray}
\label{eq:26}
\overline{\pi}^L(p_2)\!= \!\min\left\lbrace q\in \R : H^L( p_2e_2+qe_1)=H^{-,1,L}( p_2e_2+qe_1)=E(p_2) \right\rbrace,\\
\label{eq:27}
\widehat{\pi}^L(p_2)\!= \!\max\left\lbrace q\in \R : H^L(p_2e_2+qe_1)=H^{-,1,L}( p_2e_2+qe_1)=E(p_2) \right\rbrace,\\
\label{eq:37}
\overline{\pi}^R(p_2)\!= \!\min\left\lbrace q\in \R : H^R( p_2e_2+qe_1)=H^{+,1,R}( p_2e_2+qe_1)=E(p_2) \right\rbrace,\\
\label{eq:38}
\widehat{\pi}^R(p_2)\!= \!\max\left\lbrace q\in \R : H^R( p_2e_2+qe_1)=H^{+,1,R}( p_2e_2+qe_1)=E(p_2) \right\rbrace.
\end{eqnarray}
From the  convexity  of the Hamiltonians $H^i$, we deduce that for $i=L,R$, if $E^i_0(p_2)<E (p_2)$,  then 
$\overline{\pi}^i(p_2)=\widehat{\pi}^i(p_2)$. In this case, we will use the notation
\begin{equation}
\label{eq:41}
\pi^i(p_2)=\overline{\pi}^i(p_2)=\widehat{\pi}^i(p_2).
\end{equation}

\begin{lemma}
  \label{sec:global-cell-problem-2}
For any   $p_2 \in \R$:
\begin{itemize}
\item   if $E(p_2)>E_0^R(p_2)$, 
then  $\psi( p_2, \cdot)$ is affine in the interval $(1,+\infty)$ and $\partial_y \psi ( p_2, y)= \pi^R(p_2)$
\item  if $E( p_2)>E_0^L(p_2)$,
 then $\psi( p_2, \cdot)$ is affine  in the interval $(-\infty,-1)$ and $\partial_y \psi ( p_2, y)= \pi^L(p_2)$.
\end{itemize}
\end{lemma}
\begin{proof}
  If $E(p_2)>E_0^R(p_2)$, we prove, exactly as  Proposition \ref{cor:slopes_omega}  that  there exist  $\rho^*=\rho^*(p_2)  >0$ and $M^*= M^*(p_2)   \in \R$
such that, for all $y\in [\rho^*,+\infty)$, $h_1\ge 0$,
\begin{equation}
\label{eq:42}
\psi(p_2, y+h_1e_1)-\psi(p_2, y)\geq  \pi^R(p_2) h_1-M^*.
\end{equation}
From (\ref{eq:42}), classical arguments on viscosity solutions of one-dimensional equations with convex Hamiltonians yield the desired result for $y>1$.
The same kind of arguments are used for $y<-1$. 
\end{proof}
\begin{proposition}\label{sec:proof-refeq:18}
  The constant $E(p_2)$ defined in (\ref{eq:21}) satisfies (\ref{eq:18}).
\end{proposition}
\begin{proof}
From the fourth and fifth equations in (\ref{eq:22}), we see that $E(p_2)\ge \max( E^{L,M}(p_2), E^{M,R}(p_2))$. Moreover, we know that
 $ E^{M,R}(p_2)\ge  E^{M,R}_0(p_2)= \max (E^{M}_0(p_2), E^{R}_0(p_2))$ from Proposition~\ref{cor:E_bigger_than_E_0}. Similarly $ E^{L,M}(p_2)\ge  E^{L,M}_0(p_2)= \max (E^{L}_0(p_2), E^{M}_0(p_2))$.
\\
We make out two main cases:
\begin{enumerate}
\item If  $E(p_2)= E_0^M(p_2)$, then using the observations above,
we get that  $ E(p_2)=E^{L,M}(p_2)= E^{M,R}(p_2) )$, which implies (\ref{eq:18}).
\item If  $E(p_2)> E_0^M(p_2)$, then we can define two real numbers  $\pi^{M, -}< \pi^{M,+}$ such that 
  \begin{displaymath}
    \begin{split}
      H^M( \pi^{M, -} e_1+ p_2e_2 )=  H^{-,1,M}(\pi^{M, -} e_1+ p_2e_2 )= E(p_2),\\
      H^M( \pi^{M, +} e_1+ p_2e_2 )=  H^{+,1,M}( \pi^{M, +} e_1+ p_2e_2 )= E(p_2), 
    \end{split}
  \end{displaymath}
and one and only one of the following three assertions is true:
\begin{enumerate}
\item the function $\psi(p_2,\cdot)$ defined in Proposition~\ref{sec:state-constr-probl-5} is affine in $(-1,1)$ with slope $\pi^{M, -}$:  in this case, $   H^{+,1,M}( \partial_y \psi ( p_2, 1^-)  e_1+ p_2e_2 )< E(p_2)$:
using the fifth equation in (\ref{eq:22}), we deduce that \[\max \left(E^{M,R}(p_2),  H^{-,1,R}\left( \frac{d \psi }{dy} (1^+) e_1+p_2e_2  \right)\right)=E( p_2);\]
there are two subcases:
\begin{enumerate}
\item if $E(p_2)=E_0^R(p_2)$, then using the fact that $E^{M,R}(p_2)\ge E_0^R(p_2)$, we get that $E^{M,R}(p_2)= E( p_2)$
\item if $E(p_2)>E_0^R(p_2)$, then as a consequence  of Lemma~\ref{sec:global-cell-problem-2}, we see that  $H^{-,1,R} (\partial_y \psi ( p_2, 1^+) e_1 +p_2e_2)< E(p_2)$,
which again implies that $E^{M,R}(p_2)= E( p_2)$.
\end{enumerate}
Therefore  $E^{M,R}(p_2)= E( p_2)$, and since  $E^{L,M}(p_2)\le E( p_2)$ from the fourth equation in (\ref{eq:22}), we obtain (\ref{eq:18}).
\item $\psi(p_2,\cdot)$ is affine in $(-1,1)$ with slope $\pi^{M, +}$. The same arguments as in the previous case yield  that $E^{L,M}(p_2)= E(p_2)$ then  (\ref{eq:18}).
\item  $\psi(p_2,\cdot)$ is piecewise  affine in $(-1,1)$, with the slope  $\pi^{M, +}$ in  $(-1, c)$ and the slope $\pi^{M, -}$ in  $( c,1)$, for some $c$ with $|c|<1$.
Hence, $H^{-,1,M} ( \partial_y \psi ( p_2, -1^+) e_1 +p_2e_2)< E(p_2)$ and  $   H^{+,1,M}( \partial_y \psi ( p_2, 1^-)  e_1+ p_2e_2 )< E(p_2)$: therefore,
\begin{eqnarray}
  \label{eq:35}
    \ds \max\left(E^{L,M}(p_2),  H^{+,1,L}(  \partial_y \psi ( p_2, -1^-)  e_1+ p_2e_2 )\right)=E(p_2),\\
\label{eq:36}
    \ds \max\left(E^{M,R}(p_2), H^{-,1,R} ( \partial_y \psi (p_2, 1^+)  e_1 +p_2e_2)  \right) =E(p_2).
\end{eqnarray}
\begin{enumerate}
\item If  $E( p_2)= E_0^R(p_2)$, then the  very first  observation in the proof imply that $E^{M,R}(p_2)=E(p_2)$.
\item If  $E( p_2)> E_0^R(p_2)$, then from Lemma~\ref{sec:global-cell-problem-2}, $ H^{-,1,R} (\partial_y \psi ( p_2, 1^+)  e_1+p_2e_2)< E(p_2)$, and (\ref{eq:36}) yields 
that $E^{M,R}(p_2)=E(p_2)$.
\end{enumerate}
Similarly, using (\ref{eq:35}), we find that $E^{L,M}(p_2)=E(p_2)$, so \\ $ E^{L,M}(p_2)=E^{M,R}(p_2)=E(p_2)$, which yields  (\ref{eq:18}).
\end{enumerate}
\end{enumerate}
\end{proof}
\begin{remark}\label{sec:proof-refeq:18-1}
  We have actually proved that  $E(p_2)$ defined by (\ref{eq:18}) is the 
unique constant such that the global cell problem (\ref{eq:22}) has  a Lipschitz continuous solution.
\end{remark}

\subsubsection{End of the proof of Theorem~\ref{sec:main-result-2}}
\label{sec:end-proof-theorem}
The end of the proof of Theorem ~\ref{sec:main-result-2} is completely similar to the proof of the main result in \cite{MR3565416}. The general method was first proposed in \cite{MR3441209} and 
uses Evans' method of perturbed test-functions with the particular test-functions proposed in \cite{imbert:hal-01073954}, to which the  correctors found in
 Proposition \ref{sec:state-constr-probl-5} are associated. For  brevity, we do not repeat the proof here.

\section{Simultaneous passage to the limit as $\eta=\epsilon\to 0$}
\label{sec:simult-pass-limit}

\subsection{Main result}
\label{sec:main-result-3}
We now turn our attention to the case when $\eta=\epsilon$. We are interested in the asymptotic behavior of the sequence 
 $v_{\epsilon, \epsilon}$ as $\epsilon\to 0$.  The main result tells us that the limit is the same function $v$ as the one defined
in Theorem \ref{sec:main-result-2}, i.e. obtained by two successive passages to the limit  in $v_{\eta, \epsilon}$, first by letting  $\epsilon\to 0$  then $\eta\to 0$.
\begin{theorem}\label{sec:main-result-4}
   As $\epsilon\to 0$, $v_{\epsilon,\epsilon}$ converges locally uniformly  to $v$, the unique bounded viscosity solution of 
(\ref{eq:14})-(\ref{eq:17}), with  $H^{L,R}$ given by (\ref{eq:25}) and the effective flux-limiter 
$E(p_2)$ given by (\ref{eq:18}).
\end{theorem}

\begin{remark}
  \label{sec:main-result-5}
Note that the same convergence result holds for the sequence $v_{\epsilon, \epsilon^q}$ where $q$ is any positive number.
\end{remark}
\subsection{Correctors}\label{sec:correctors}
Let us consider the problem in the original geometry dilated by the factor $1/\epsilon$.
Defining $\Omega_{1,\epsilon}^L, \Omega_{1,\epsilon}^R,  \Gamma_{1,\epsilon}$ and  $ H_{ \Gamma_{1,\epsilon}}$ as in \S~\ref{sec:geometry} and \S~\ref{sec:optim-contr-probl}, and recalling that $Y^\rho=\{y\in \R^2: |y_1|<\rho\}$, we consider the truncated cell problem 
\begin{equation}
\label{eq:43}
\left\{
    \begin{array}[c]{lll}
    H^L(Du(y)+p_2e_2)&\le E_{\epsilon,\rho}( p_2)&\hbox{ if } y\in   \Omega_{1,\epsilon}^L \cap  Y^\rho   ,  \\
    H^L(Du(y)+p_2e_2)&\ge E_{\epsilon,\rho}( p_2)&\hbox{ if } y\in    \Omega_{1,\epsilon}^L \cap  \overline{Y^\rho}   ,  \\
    H^R(Du(y)+p_2e_2)&\le E_{\epsilon,\rho}( p_2)&\hbox{ if } y\in   \Omega_{1,\epsilon}^R \cap  Y^\rho   ,  \\
    H^R(Du(y)+p_2e_2)&\ge E_{\epsilon,\rho}( p_2)&\hbox{ if } y\in    \Omega_{1,\epsilon}^R \cap  \overline{Y^\rho}   ,  \\
    H_{ \Gamma_{1,\epsilon}} (y,Du^L(y)+p_2e_2 ,   Du^R(y)+p_2e_2)&=
 E_{\epsilon,\rho}( p_2)&\hbox{ if } y\in    \Gamma_{1,\epsilon}    ,  \\
    u \hbox{ is $\epsilon$ periodic w.r.t. } y_2,
    \end{array}
\right.
\end{equation}
where $\rho$ is large enough such that $\Gamma_{1,\epsilon}\subset\subset Y^\rho$ and the inequations are understood in the sense of viscosity.  The following lemma can be proved with the same ingredients as in \S~\ref{sec:state-constr-probl-1}:
\begin{lemma}
\label{sec:trunc-cell-probl-1}
There is a unique  $E_{\epsilon,\rho}( p_2)\in \R$ such that (\ref{eq:43}) admits a viscosity solution. 
For this choice of  $E_{\epsilon,\rho}( p_2)$, there exists a  solution  $\xi_{\epsilon,\rho}(p_2,\cdot)$
 which is Lipschitz continuous with a Lipschitz constant  $L$ depending on $p_2$ only (independent of $\epsilon$ and $\rho$).
\end{lemma}
As in \cite{MR3299352,MR3565416},  using the optimal control interpretation of (\ref{eq:43}), it is easy to prove that for 
a positive $K$ which may depend on  $p_2$ but not on $\rho$ and $\epsilon$ and for 
all $0<\rho_1\le \rho_2$, 
\begin{equation}
  \label{eq:61}
E_{\epsilon,\rho_1}( p_2)\leq E_{\epsilon,\rho_2}( p_2)\leq K.
\end{equation}
For $p_2\in \R$, let $E_{\epsilon}(p_2)$ be defined by
\begin{equation}
\label{eq:47}
  E_{\epsilon}( p_2)=\lim_{\rho\rightarrow \infty} E_{\epsilon,\rho}( p_2).
\end{equation}
For  a fixed  $p_2\in\R$, the {\sl global cell-problem} reads
\begin{equation}
\label{eq:48}
\left\{
    \begin{array}[c]{lll}
    H^L(Du(y)+p_2e_2)&= E_{\epsilon}( p_2)&\hbox{ if } y\in   \Omega_{1,\epsilon}^L    ,  \\
    H^R(Du(y)+p_2e_2)&= E_{\epsilon}( p_2)&\hbox{ if } y\in   \Omega_{1,\epsilon}^R    ,  \\
    H_{ \Gamma_{1,\epsilon}} (y,Du^L(y)+p_2e_2 ,   Du^R(y)+p_2e_2)&=
 E_{\epsilon}( p_2)&\hbox{ if } y\in    \Gamma_{1,\epsilon}    ,  \\
    u \hbox{ is $\epsilon$ periodic w.r.t. } y_2.
    \end{array}
\right.
\end{equation}
The following theorem can be obtained by using the same arguments as in \S~\ref{sec:passage-limit-as}:
\begin{theorem}
\label{sec:trunc-cell-probl-2}
Let $\xi_{\epsilon, \rho}( p_2,\cdot)$ be a sequence of uniformly Lipschitz continuous solutions of the truncated cell-problem (\ref{eq:43}) which converges to $ \xi_\epsilon( p_2,\cdot)$
locally uniformly on $\R^2$ as $\rho\to +\infty$. The function  $\xi_\epsilon( p_2,\cdot)$ is a Lipschitz continuous viscosity solution of the global cell-problem (\ref{eq:48}).
% and $E_{\epsilon}( p_2)\ge \max( E_0^L(p_2),E_0^R(p_2))$.
% By subtracting  $\chi(p_2, 0)$, we may also assume that $\chi(p_2,0)=0$.
\end{theorem}

Using the control interpretation of (\ref{eq:43}), we see that $ E_\epsilon(p_2)$ is bounded independently
 of  $\epsilon$.
We may thus suppose that, possibly after the extraction of a subsequence, $\lim_{\epsilon\to 0} E_\epsilon(p_2)= E(p_2)$. 
Moreover, if $E(p_2)>  E_0^R(p_2)$, then for $\epsilon$ small enough, $E_\epsilon(p_2)>  E_0^R(p_2)$  and  we can define
$\pi^R(p_2)$ and $\pi^R_\epsilon (p_2)$ as the unique real numbers such that
\begin{displaymath}
  \begin{split}
    H^R( p_2e_2+{\pi}^R(p_2)e_1)&=H^{+,1,R}( p_2e_2+{\pi}^R(p_2)e_1)=E(p_2),
    \\
    H^R( p_2e_2+{\pi}^R_\epsilon(p_2)e_1)&=H^{+,1,R}( p_2e_2+{\pi}^R_\epsilon(p_2)e_1)=E_\epsilon(p_2).
  \end{split}
\end{displaymath}
Note that ${\pi}^R(p_2)=\lim_{\epsilon\to 0} {\pi}^R_\epsilon(p_2)$.
Then we can prove exactly as Proposition~\ref{cor:slopes_omega} that  if $E( p_2)>E_0^R(p_2)$, then,  
there exist  $\rho^*=\rho^*(p_2)  >0$ and $M^*= M^*(p_2)   \in \R$ 
 such that, for all $\epsilon >0$ small enough, for all $(y_1,y_2)\in [\rho^*,+\infty)\times \R$, $h_1\ge 0$ and $h_2\in \R$,
 \begin{equation}
 \label{eq:55}
 \xi_\epsilon( p_2, y+h_1e_1+h_2 e_2)-\xi_\epsilon( p_2, y)\geq  \pi^R_\epsilon(p_2) h_1-M^*.
 \end{equation}
Of course, the same observations can be made on the left side of the interface:
if $E(p_2)>  E_0^L(p_2)$, then  for $\epsilon$ small enough, $E_\epsilon(p_2)>  E_0^L(p_2)$ and we can define
$\pi^L(p_2)$ and $\pi^L_\epsilon (p_2)$ as the unique real numbers such that
\begin{displaymath}
  \begin{split}
    H^L( p_2e_2+{\pi}^L(p_2)e_1)&=H^{-,1,L}( p_2e_2+{\pi}^L(p_2)e_1)=E(p_2),
    \\
    H^L( p_2e_2+{\pi}^L_\epsilon(p_2)e_1)&=H^{-,1,L}( p_2e_2+{\pi}^L_\epsilon(p_2)e_1)=E_\epsilon(p_2).
  \end{split}
\end{displaymath}
If $E( p_2)>E_0^L(p_2)$, then,   there exist  $\rho^*=\rho^*(p_2)  >0$ and $M^*= M^*(p_2)   \in \R$ 
 such that, for all $\epsilon >0$ small enough, for all $(y_1,y_2)\in   (-\infty,-\rho^*] \times \R$, $h_1\ge 0$ and $h_2\in \R$,
 \begin{equation}
 \label{eq:56}
 \xi_\epsilon( p_2, y+h_1e_1+h_2 e_2)-\xi_\epsilon( p_2, y)\leq  \pi^L_\epsilon(p_2) h_1+M^*.
 \end{equation}

Using similar arguments to those in \S~\ref{sec:passage-limit-as}
 and  Remark~\ref{sec:proof-refeq:18-1}, we obtain the following results:
\begin{theorem}
\label{sec:trunc-cell-probl-3}
Let $(\epsilon_n)$ be a sequence of positive numbers tending to $0$  such that the solution of (\ref{eq:48})
$(\xi_{\epsilon_n}( p_2,\cdot), E_{\epsilon_n}(p_2))$ satisfy:  
$E_{\epsilon_n} (p_2)\to E(p_2)$ and $\xi_{\epsilon_n}( p_2,\cdot)\to \xi( p_2,\cdot)$ locally uniformly.  
Then $\xi( p_2,\cdot)$ depends on $y_1$ only and is Lipchitz continuous,
 $(\xi( p_2,\cdot), E(p_2) )$ is a  solution of (\ref{eq:22}) and $ E(p_2)= \max(E^{L,M}(p_2), E^{M,R}(p_2) )$.
%  Then  $E(p_2)\ge  \max( E_0^L(p_2),E_0^R(p_2))$ and  $\xi( p_2,\cdot)$ depends on $y_1$ only,
%  and is a Lipschitz continuous viscosity solution of (\ref{eq:22}).
% % By adding a same  constant to $\xi( p_2,\cdot)$ and $\xi_{\epsilon_n}( p_2,\cdot)$, one can choose that $\xi( p_2,0)=0$.
% Moreover,   if $E(p_2)> E_0^R(p_2)$ then $\xi( p_2,\cdot)$ is affine in $\{y_1>1\}$ 
%  and $\frac {d \xi}{dy_1}(p_2,y_1)=\pi^R(p_2)$ for all $y_1>1$.
%  If $E(p_2)> E_0^L(p_2)$ then $\xi( p_2,\cdot)$ is affine in $\{y_1<-1\}$  and $\frac {d \xi}{dy_1}(p_2,y_1)=\pi^L(p_2)$ for all $y_1<-1$. 
% \\
%  Finally, 
% \begin{equation}\label{eq:57}
%   E(p_2)= \max(E^{L,M}(p_2), E^{M,R}(p_2) ).
% \end{equation}
\end{theorem}

\begin{corollary}
  \label{sec:trunc-cell-probl-4}
As $\epsilon\to 0$, the whole sequence $E_\epsilon(p_2)$ tends to $ \max(E^{L,M}(p_2), E^{M,R}(p_2) )$.
\end{corollary}
The construction of the function $\xi$ has been useful to characterize $E(p_2)$ by (\ref{eq:18}).
 However, since $\xi$ is the solution of~(\ref{eq:22}), it is not directly connected to the oscillating interface $\Gamma_{\epsilon,\epsilon}$, and $\xi$  
will not be useful when applying Evans'  method 
to prove Theorem~\ref{sec:main-result-4}.
The function used in Evans' method will rather be $\xi_\epsilon$ and the following proposition 
will therefore be useful. We skip its proof, because it is very much similar to that of  Proposition~\ref{cor:control_slopes_W}.

\begin{proposition}
\label{sec:trunc-cell-probl-5}
For any $p_2>0$, there exists a sequence
$(\epsilon_n)$ of positive numbers tending to $0$  such that  $y\mapsto \epsilon_n \xi_{\epsilon_n}( p_2,\frac y {\epsilon_n})$ converges   locally uniformly to $y\mapsto W(p_2, y)$. The function $W(p_2,\cdot)$ does not depend on $y_2$ and is a Lipschitz continuous viscosity solution of (\ref{eq:19}).   
 By adding a same  constant to $W(p_2,\cdot)$ and $\xi_{\epsilon_n}( p_2,\cdot)$, one can impose that $W( p_2,0)=0$. 
Moreover,
\begin{equation}
\label{eq:89}
- \widehat{\pi}^L(p_2) (y_1)^-
+ \overline{\pi}^R(p_2) (y_1)^+ \le W(p_2,y) \le 
- \overline{\pi}^L(p_2) (y_1)^-
+ \widehat{\pi}^R(p_2) (y_1)^+,
\end{equation}
where for $i=L,R$, the values $\overline{\pi}^i(p_2)$ and  $\widehat{\pi}^L(p_2)$  are defined in (\ref{eq:26})-(\ref{eq:38}).
\end{proposition}

\subsection{Proof of Theorem~\ref{sec:main-result-4}}
\label{sec:proof-theor-refs}
The proof of Theorem~\ref{sec:main-result-4} is similar to that of Theorem~\ref{th:convergence_result}.
Let us  consider the relaxed semi-limits
 \begin{equation}
\label{eq:90}
 \overline{v}(z)={\limsup_\epsilon}^{*} {v}_{\epsilon,\epsilon}(z)=\limsup_{z'\to z, \epsilon\to 0}{v}_{\epsilon,\epsilon}(z')
 \quad \mbox{ and } \quad \underline{v}(z)=\underset{\epsilon}{{\liminf}_{*}}{v}_{\epsilon,\epsilon}(z)
=\liminf_{z'\to z, \epsilon\to 0}{v}_{\epsilon,\epsilon}(z').
 \end{equation}
Note that  $ \overline{v}$ and $ \underline{v}$ are well defined, 
since  $\left( v_{\epsilon,\epsilon}\right)_\epsilon$ is uniformly bounded.
It is classical to check that the functions $\overline{v}(z)$ and $\underline{v}(z)$
 are respectively a bounded subsolution and a bounded supersolution in $\Omega^i$, $i=L,R$, of
\begin{equation}
\label{eq:91}
\l u(z)+H^i(Du(z))= 0.
\end{equation}
 We will prove that $\overline{v}$  and $\underline{v}$  are respectively a subsolution and a supersolution
 of (\ref{eq:16}). From the comparison theorem proved in \cite{barles2013bellman,imbert:hal-01073954,oudet2014}, 
this will imply that  $\overline{v}=\underline{v}=v=\lim_{\epsilon\to 0} v_{\epsilon,\epsilon}$.  We just have to check the transmission condition (\ref{eq:17}).
\\
Take $\bar z =(0, \bar z_2) \in \Gamma$. It is possible to use the counterpart of Theorem \ref{th:restriction_set_of_test_functions} 
because $\overline{v}$ is Lipschitz continuous, see  Remark~\ref{sec:reduced-set-test-1}.
\\
Take a  test-function  of the form 
\begin{equation}
\label{eq:92}
\varphi(z+ t e_1)= \psi(z)+   \left( \overline \pi^R\left( \partial_{z_2}\psi(\bar z) \right) 1_{t>0}+
\widehat \pi^L \left(  \partial_{z_2}\psi(\bar z)\right) 1_{t<0} \right) t, 
\quad \forall z\in  \Gamma , t\in \R,
\end{equation}
for  a $\cC^1$ function $\psi: \Gamma \to \R$,  such that $\overline{v}-\varphi$ has a strict local maximum at $\bar z$ and that $\overline{v}(\bar z)=\varphi(\bar z)$. \\
 Let us argue by contradiction and assume that 
\begin{equation}
\label{eq:93}
\lambda \varphi(\bar z)+ \max\left(E(\partial_{z_2}\varphi(\bar z)), H^{L,R}(D \varphi^L(\bar z),D \varphi^R(\bar z)) \right)=\theta >0.
\end{equation}
From  (\ref{eq:92}),  we see that $
 H^{L,R}(D \varphi^L(\bar z),D \varphi^R(\bar z))\le E (\partial_{z_2}\varphi(\bar z))$ and
 (\ref{eq:93}) is equivalent to 
\begin{equation}
\label{eq:94}
\lambda \psi(\bar z)+E(   \partial_{z_2}\psi(\bar z) )=\theta>0.
\end{equation}
\paragraph{Step 1}
Consider a sequence $(\xi_{\epsilon_n})_n$ as in  Proposition~\ref{sec:trunc-cell-probl-5},
 that we note $(\xi_\epsilon)$ for short.
 We claim that for $\epsilon$ and $r$ small enough, the function $\varphi^\epsilon$:
 \[\varphi^\epsilon(z)=\psi(z_2 e_2)+\epsilon\xi_\epsilon( \partial_{z_2}\psi(\bar z),\frac {z} {\epsilon})\] 
is a viscosity supersolution of 
\begin{equation}
\label{eq:95}
\left\{
    \begin{array}[c]{rcll}
   \lambda \varphi^\epsilon(z)+ H^i(D\varphi^\epsilon(z))& \ge& \frac{\theta}{2}\quad &\hbox{ if } z\in\Omega^i_{\epsilon,\epsilon}\cap B(\bar z,r),\; i=L,R, \\
   \lambda \varphi^\epsilon(z)+H_{\Gamma_{\epsilon, \epsilon}} (z,D\left( \varphi^\epsilon\right)^L(z),D\left( \varphi^\epsilon\right)^R(z)) &\ge& \frac{\theta}{2}&\hbox{ if } z\in\Gamma_{\epsilon,\epsilon}\cap B(\bar z,r).
    \end{array}
\right.
\end{equation}
Indeed, if $\nu$ is a test-function in $\cR_{\epsilon,\epsilon}$ such that $\varphi^\epsilon-\nu$ has a local minimum at
  $z^\star\in B(\bar z,r)$, then, from the definition of $\varphi^\epsilon$,
 $y\mapsto \xi_\epsilon( \partial_{z_2}\psi(\bar z),y)-\frac{1}{\epsilon}\left(\nu(\epsilon y)
-\psi(\epsilon y_2 e_2) \right)$
 has a local minimum at $\frac{z^\star}{\epsilon}$.
\\
If $\frac{z^\star}{\epsilon}\in \Omega^i_{1,\epsilon}$, for $i=L$ or  $R$, then  , from (\ref{eq:48}), 
$H^i(D\nu(z^\star) -\partial_{z_2} \psi( z^\star_2 e_2) e_2  + \partial_{z_2}\psi(\bar z) e_2)
\ge E_\epsilon(\partial_{z_2}\psi(\bar z))$.
From the regularity properties of   $H^i$,
 \begin{displaymath}
H^i(D\nu(z^\star) -\partial_{z_2} \psi(z^\star_2 e_2) e_2  + \partial_{z_2}\psi(\bar z) e_2)
= H^i(D\nu(z^\star)) + o_{r\to 0}(1),
\end{displaymath}
thus, using also the convergence of $E_\epsilon$ to $E$,
\begin{equation*}
\lambda \varphi^\epsilon(z^\star)+ H^i(D\nu(z^\star)) 
\ge E(\partial_{z_2}\psi(\bar z)) +\lambda
\left(\psi(z^\star_2 e_2)+\epsilon \xi_\epsilon( \partial_{z_2}\psi(\bar z),\frac {z^\star} {\epsilon})\right)
+o_{r\to 0}(1) +o_{\epsilon\to 0}(1).  
\end{equation*}
From (\ref{eq:94}), this implies that 
\begin{equation*}
\lambda \varphi^\epsilon(z^\star)+ H^i(D\nu(z^\star))\ge \theta +\lambda \epsilon\xi_\epsilon( \partial_{z_2}\psi(\bar z),\frac {z^\star} {\epsilon})+o_{r\to 0}(1)+o_{\epsilon\to 0}(1).
\end{equation*}
From the  Lipschitz continuity of $\xi_\epsilon( \partial_{z_2}\psi(\bar z),\cdot)$ with a constant independent of $\epsilon$, we get that 
$ \epsilon\xi_\epsilon( \partial_{z_2}\psi(\bar z),\frac {z^\star} {\epsilon})=  \epsilon\xi_\epsilon( \partial_{z_2}\psi(\bar z),\frac {\bar z} {\epsilon}) +o_{r\to 0}(1)$. Moreover it is easy to check that
$ \epsilon\xi_\epsilon( \partial_{z_2}\psi(\bar z),\frac {\bar z} {\epsilon})= o_{\epsilon\to 0}(1)$.
Therefore, for $r$ and $\epsilon$ small enough, $\lambda \varphi^\epsilon(z^\star)+ H^i(D\nu(z^\star))\ge \theta/2$.
\\
If $\frac{ z^\star }{\epsilon}\in  \Gamma_{1,\epsilon}$, then we have 
\[H^{+,L}_{\Gamma_{1,\epsilon}} ( \frac{ z^\star }{\epsilon}, D\nu^L (z^\star) -\partial_{z_2} \psi(z^\star_2 e_2) e_2  + \partial_{z_2}\psi(\bar z) e_2,   )\ge E_\epsilon(\partial_{z_2}\psi(\bar z))\]
or 
\[H^{+,R}_{\Gamma_{1,\epsilon}} ( \frac{ z^\star }{\epsilon}, D\nu^R (z^\star) -\partial_{z_2} \psi(z^\star_2 e_2) e_2  + \partial_{z_2}\psi(\bar z) e_2,   )\ge E_\epsilon(\partial_{z_2}\psi(\bar z))
\]
Since the Hamiltonians $H^{\pm,i}_{\Gamma_{1,\epsilon}}$ enjoys the same regularity properties as  $H^{\pm i}$,
 it is possible to use the same arguments as  in the case when $\frac{z^\star}{\epsilon}\in \Omega^i_{1,\epsilon}$.
 For  $r$ and $\epsilon$ small enough, 
\[\lambda \varphi^\epsilon(z^\star)+H_{\Gamma_{\epsilon, \epsilon}} (z^\star,D\left( \varphi^\epsilon\right)^L(z^\star),D\left( \varphi^\epsilon\right)^R(z^\star)) \ge \frac{\theta}{2}.
   \] 
The claim that $\varphi^\epsilon$  is a supersolution of (\ref{eq:95}) is proved.
\paragraph{Step 2} Let us prove that there exist some positive constants $K_r>0$ and  $\epsilon_0>0$ such that
\begin{equation}
\label{eq:99}
 v_{\epsilon,\epsilon}(z)+K_r\le \varphi^\epsilon(z), \quad \forall z\in \partial B(\bar z, r), \;\forall \epsilon\in (0, \epsilon_0).
\end{equation}
Indeed, since $\overline{v}-\varphi$ has a strict local maximum at $\bar z$ 
and since $\overline{v}(\bar z)=\varphi(\bar z)$, 
there exists a positive constant $\tilde K_r>0$ such that 
$\overline{v}(z)+\tilde K_r\le \varphi(z)$ for any $z\in \partial B(\bar z, r)$.
Since $\ds \overline{v}={\limsup_\epsilon}^{*} v_{\epsilon,\epsilon}$,  there exists $\tilde \epsilon_0>0$ such that 
\begin{equation}
\label{eq:96}
v_{\epsilon,\epsilon}(z)+\frac{\tilde K_r}{2}\le \varphi(z)\quad \hbox{
for any $0<\epsilon<\tilde \epsilon_0$ and $z\in \partial B(\bar z, r)$}.
\end{equation}
On the other hand, from  Proposition~\ref{sec:trunc-cell-probl-5},
\begin{equation}
\label{eq:98}
   \psi(z_2 e_2)+W( \partial_{z_2}\psi(\bar z),z) \ge
 \psi(z_2 e_2)+   \left( \overline \pi^R   ( \partial_{z_2}\psi(\bar z) ) 1_{z_1>0}
+\widehat \pi^L( \partial_{z_2}\psi(\bar z)) 1_{z_1<0} \right) z_ 1= \varphi(z).
\end{equation}
Moreover,  $z\mapsto \varphi^\epsilon(z)$ converges locally uniformly to
 $z\mapsto \psi(z_2e_2)+W( \partial_{z_2}\psi(\bar z),z)$ as $\epsilon$ tends to $0$.
By collecting the latter observation, (\ref{eq:98}) and (\ref{eq:96}), we get \eqref{eq:99} 
for some constants $K_r>0$ and $\epsilon_0>0$.
\paragraph{Step 3}
From the previous steps, we find by comparison that for $r$ and $\epsilon$ small enough, 
\begin{displaymath}
 v_{\epsilon,\epsilon}(z)+K_r\le \varphi^\epsilon(z) \quad \quad \forall z \in B(\bar z, r).
\end{displaymath}
Taking the $\limsup$ as $z=\bar z$ and $\epsilon\to 0$,  we obtain
\begin{displaymath}
\overline{v}(\bar z)+K_r\le \psi(\bar z)=\varphi(\bar z)=\overline{v}(\bar z),
\end{displaymath}
which cannot happen. The proof is completed. \qed

\appendix
 \section{Proofs of Propositions \ref{cor:slopes_omega} and \ref{cor:control_slopes_W}}
\label{sec:proofs-prop-refc}

\begin{lemma}[Control of slopes on the truncated domain]
\label{SLOPESLemma1}
 With $E^{M,R}$ and $E^R_0$ respectively defined in~(\ref{eq:def_E}) and (\ref{def:E_0^i}),  let  $p_2\in\R$ be such that 
$E^{M,R}(p_2)>E_0^R(p_2)$.
There exists $\rho^*=\rho^*( p_2)>0$, $\delta^*=\delta^*(p_2)>0$, $m(p_2,\cdot):  [\rho^*,+\infty)  \times [0, \delta^*]   
 \to \R_+$ satisfying \\
$\lim_{\delta \to 0^+} \lim_{ \rho \to +\infty} m(p_2, \rho,\delta)=0$ and $M^*=M^*(p_2)$, such that for all
  $\delta\in (0,\delta^*]$, $\rho\ge \rho^*$, $(y_1,y_2)\in [\rho^*,\rho]\times \R$, $h_1\in [0,\rho-y_1]$ and $h_2\in \R$,
\begin{equation}
\label{slope3}
\chi_\rho( p_2, y+h_1e_1+h_2 e_2)-\chi_\rho( p_2, y)\geq  (\Pi^R(p_2)-m(p_2,\rho,\delta)) h_1-M^*,
\end{equation}
where 
$\Pi^R(p_2)$ is  given by \eqref{eq:special_notation_for_Pi} and 
$\chi_\rho(p_2,\cdot)$ is a solution of \eqref{trunc-cellp} given by Lemma \ref{lem:existence_solution_truncated_cell_pb}.

\medskip

 Similarly, let  $p_2\in\R$ be such that $E^{M,R}( p_2)>E_0^M(p_2)$. 
There exists $\rho^*>0$, $\delta^*>0$, $m(p_2,\cdot)$ and $M^*$ as above, such that
 for all  $\delta\in (0,\delta^*]$, $\rho\ge \rho^*$, $(y_1,y_2)\in [-\rho,-\rho^*]\times \R$, $h_1\in [0,\rho+y_1]$ and $h_2\in \R$,
\begin{equation}
\label{slope3_bis}
\chi_\rho(p_2, y-h_1e_1+h_2 e_2)-\chi_\rho( p_2, y)\geq  -(\Pi^M(p_2)+m(p_2,\rho,\delta)) h_1-M^*.
\end{equation}

\end{lemma}

\begin{proof}
Let us focus on \eqref{slope3_bis}, since the proof of \eqref{slope3} is similar and even simpler.
 Recall that $\rho\mapsto \lambda_\rho(p_2)$ is nondecreasing and tends to $E^{M,R}(p_2)$
 as $\rho\to+\infty$. Choose $\rho^*=\rho^*(p_2)>0$ s.t.  $E^{M,R}( p_2)>\lambda_\rho(p_2)>E_0^M(p_2)$ for any $\rho\ge \rho^*$.
 Then, choose $\delta^*=\delta^*(p_2)>0$ s.t.  $\lambda_\rho(p_2)-\delta>E_0^M(p_2)$ for any $\delta\in (0,\delta^*]$ and $\rho\ge \rho^*$.
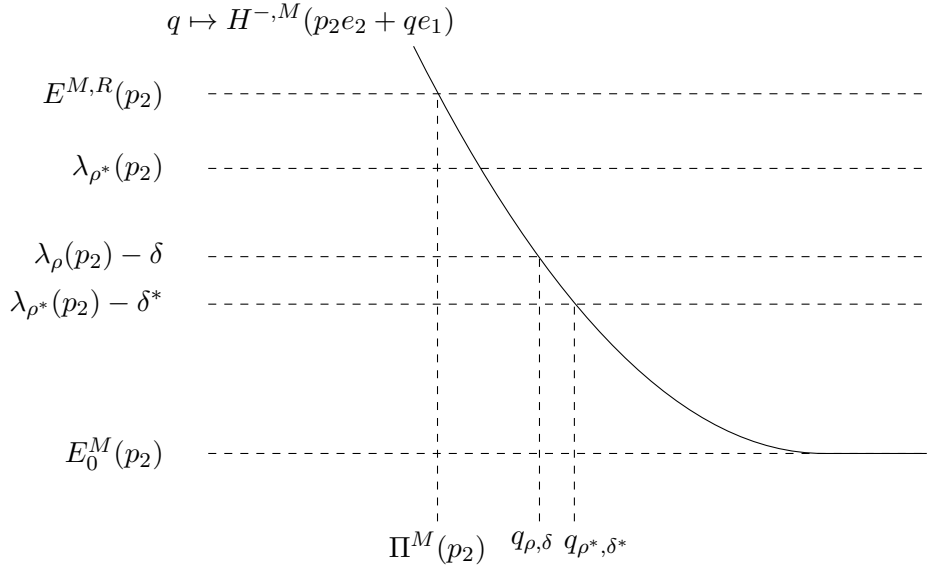
\begin{figure}[H]
\begin{center}
\begin{tikzpicture}[scale=0.90]
\draw (7.5,0) -- (9,0);
\draw (1.5,6) .. controls (3.5,2) and (5.5,0) .. (7.5,0);
\draw (0,6) node[right,above]{$q\mapsto H^{-,M}(p_2e_2+q e_1)$} ;
 \draw[dashed] (-1.5,5.3) -- (9,5.3);
 \draw (-2,5.3) node[left]{$E^{M,R}(p_2)$};
 \draw[dashed] (-1.5,0) -- (9,0);
 \draw (-2,0) node[left]{$E_0^M(p_2)$};
 \draw[dashed] (-1.5,4.2) -- (9,4.2);
 \draw (-2,4.2) node[left]{$\lambda_{\rho^*}(p_2)$};
 \draw[dashed] (-1.5,2.2) -- (9,2.2);
 \draw (-2,2.2) node[left]{$\lambda_{\rho^*}(p_2)-\delta^*$};
 \draw[dashed] (-1.5,2.9) -- (9,2.9);
 \draw (-2,2.9) node[left]{$\lambda_{\rho}(p_2)-\delta$};
\draw[dashed] (1.85,-1) -- (1.85,5.3);
 \draw (1.85,-1) node[below]{$\Pi^M(p_2)$};
 \draw[dashed] (3.34,-1) -- (3.34,2.9);
 \draw (3.34,-1) node[below]{$\!\!\!q_{\rho,\delta}$};
 \draw[dashed] (3.85,-1) -- (3.85,2.2);
 \draw (3.85,-1) node[below]{$\phantom{++}q_{\rho^*,\delta^*}$};
\end{tikzpicture}
\caption{Construction of $\rho^*,\delta^*$ and $q_\delta$: here $\lambda_\rho(p_2)-\delta<\lambda_{\rho^*}(p_2)$ but the opposite situation is possible. }
\label{fig:proof_slop_properties}
\end{center}
\end{figure}

Let us fix $\rho> \rho^*$, $\delta\in (0,\delta^*]$ and $\bar y=(\bar y_1,\bar y_2)\in [-\rho,-\rho^*]\times \R$.
 Consider $y \mapsto \chi_\rho(p_2,y)$ a solution of \eqref{trunc-cellp} as in Lemma \ref{lem:existence_solution_truncated_cell_pb}. The function $\chi_\rho(p_2,\cdot)$ is $\eta$-periodic 
with respect to $y_2$ and Lipschitz continuous with constant $L=L(p_2)$. Thus, for any $(y_1,y_2)\in \{\bar y_1\}\times \R$, 
\begin{equation*}
\chi_ \rho (p_2,y)-\chi_\rho (p_2,\bar y) \geq -L\eta.
\end{equation*}
Let us define
\begin{equation}
\label{eq:101}
   \tilde v(y)=\chi_\rho(p_2,y)-\chi_\rho (p_2,\bar y).
\end{equation}
It is a supersolution of 
\begin{equation}
\label{eq:proof_lem_control_slopes_bounded_value_pb_super}
\left\{
    \begin{array}[c]{lll}
    H^L(D\tilde v(y)+p_2e_2)&\ge \lambda_\rho( p_2)&\hbox{ if } y\in   \widetilde \Omega_{\eta}^L  \hbox{ and }-\rho\le  y_1< \bar y_1,  \\
    H^R(D\tilde v(y)+p_2e_2)&\ge \lambda_\rho( p_2)&\hbox{ if } y\in   \widetilde \Omega_{\eta}^R  \hbox{ and }-\rho\le y_1< \bar y_1,  \\
    H_{ \widetilde \Gamma_{\eta}} (D\tilde v^L(y)+p_2e_2 ,   D\tilde v^R(y)+p_2e_2,y)&\ge \lambda_\rho(p_2)&\hbox{ if } y\in   \widetilde \Gamma_{\eta} \hbox{ and }-\rho\le y_1< \bar y_1, \\
    \tilde v(y)\ge  -L\eta, & &\hbox{ if } y_1=\bar y_1,
   \\
    \tilde v \hbox{ is 1-periodic w.r.t. } y_2/\eta.
    \end{array}
\right.
\end{equation}
On the other hand, since $\rho\ge \rho^*$  and $\delta\in (0,\delta^*]$, there exists a unique $q_{\rho,\delta}\in \R$, see Figure \ref{fig:proof_slop_properties}, such that 
\begin{equation}
\label{slope1}
\lambda_\rho(p_2)-\delta=H^M(p_2e_2+q_{\rho,\delta} e_1)=H^{-,1,M}(p_2e_2+q_{\rho,\delta} e_1).
\end{equation}
Observe that $q_{\rho^*,\delta^*}\ge q_{\rho,\delta^*}\ge q_{\rho,\delta} \ge \Pi^M(p_2)$ and that 
$\lim_{\delta\to 0+} \lim_{\rho\to +\infty} q_{\rho,\delta} =\Pi^M(p_2)$. Choose $ m(p_2,\rho,\delta)=q_{\rho,\delta}-\Pi^M(p_2)\ge 0$ and consider the function $w$:
\begin{equation}
  \label{eq:102}
w(y)=   q_{\rho,\delta} y_1+ \zeta( q_{\rho,\delta} e_1+ p_2 e_2,y_2).
\end{equation}
It satisfies 
\begin{equation}
\label{eq:11}
\left\{
    \begin{array}[c]{ll}
    H^L(Dw(y)+p_2e_2)=  H^M( q_{\rho,\delta} e_1+ p_2 e_2)= \lambda_\rho( p_2)-\delta,\quad &\hbox{ if } y\in   \widetilde \Omega_{\eta}^L  \hbox{ and }-\rho<  y_1< \bar y_1,  \\
    H^R(Dw(y)+p_2e_2)=  H^M( q_{\rho,\delta} e_1+ p_2 e_2)= \lambda_\rho( p_2)-\delta,\quad  &\hbox{ if } y\in   \widetilde \Omega_{\eta}^R  \hbox{ and }-\rho < y_1< \bar y_1,  \\
    H_{ \widetilde \Gamma_{\eta}} (Dw^L(y)+p_2e_2 ,   Dw^R(y)+p_2e_2,y)=  H^M( q_{\rho,\delta} e_1+ p_2 e_2)&= \lambda_\rho( p_2)-\delta \\
%\quad \quad \hfill 
&\hbox{ if } y\in   \widetilde \Gamma_{\eta} \hbox{ and }-\rho< y_1< \bar y_1.
\end{array}
\right.
\end{equation}
\begin{remark}
  \label{sec:proofs-prop-refc-1}
Note that $w$ also satisfies 
\begin{equation}
\label{eq:12}
\left\{
    \begin{array}[c]{ll}
    H^{-,1,L}   (Dw(y)+p_2e_2)\le  H^M ( q_{\rho,\delta} e_1+ p_2 e_2)= \lambda_\rho( p_2)-\delta,\quad &\hbox{ if } y\in   \widetilde \Omega_{\eta}^L  \hbox{ and }  y_1=-\rho,  \\
    H^{-,1,R}(Dw(y)+p_2e_2)\le H^M( q_{\rho,\delta} e_1+ p_2 e_2)= \lambda_\rho( p_2)-\delta,\quad  &\hbox{ if } y\in   \widetilde \Omega_{\eta}^R  \hbox{ and }  y_1=-\rho,  \\
    H^{-,1} _{ \widetilde \Gamma_{\eta}} (Dw^L(y)+p_2e_2 ,   Dw^R(y)+p_2e_2,y)\le  H^M( q_{\rho,\delta} e_1+ p_2 e_2)&= \lambda_\rho( p_2)-\delta \\
%\quad \quad \hfill 
&\hbox{ if } y\in   \widetilde \Gamma_{\eta} \hbox{ and } y_1= -\rho,
\end{array}
\right.
\end{equation}
in the sense of viscosity, where $ H^{-,1} _{ \widetilde \Gamma_{\eta}}  (p ,  q ,y) $ is defined for $y_1<0$ and $p,q\in \R^2$ such that $p_1=q_1$  as the nonincreasing part  of 
 $p_1\mapsto H _{ \widetilde \Gamma_{\eta}}     (p_1e_1+p_2e_2 ,  p_1e_1+q_2 e_2 ,y) $.
\end{remark}
Moreover,  for any $(y_1,y_2)\in \{\bar y_1\}\times \R$,
\begin{displaymath}
w(y)-q_{\rho,\delta} \bar y_1\le C= C(p_2)
\end{displaymath}
so  the function $u$ defined on  $ [-\rho,\bar y_1]\times \R$ by
\begin{equation}
  \label{eq:103}
u(y)= w(y)-q_{\rho,\delta} \bar y_1- C-L\eta
\end{equation}
  is a subsolution  of (\ref{eq:11}),~(\ref{eq:12}) and is such that $u(\bar y_1, \cdot)\le -L\eta$.
\\
By a comparison result whose proof is sketched below, for all $y\in  [-\rho,\bar y_1]\times \R$,
\begin{equation}
\label{eq:proof_lem_control_slopes_bounded_value_pb_end}
%\begin{array}{lll} 
v(y)\ge u(y)
%\chi_\rho(p_2,y)-\chi_\rho (p_2,\bar y) \geq u^{M}(y)
% & = & q_{\rho,\delta}(y_1-\bar y_1) +\zeta( q_{\rho,\delta} e_1+ p_2 e_2,y_2)- C-L\eta \\
%& \geq & 
\ge (\Pi^M(p_2)+m(\rho, \delta)) ( y_1-\bar y_1)-M^*,
%\end{array}
\end{equation}
  where $M^*$ is a constant depending only of $x$ and $p_2$. This is the 
 desired result. There remains to prove the comparison result.

\paragraph{Proof of (\ref{eq:proof_lem_control_slopes_bounded_value_pb_end})}
Call $m=\max_{-\rho \le y_1\le \bar y_1} (u(y)-v(y))$ and assume by contradiction that $m>0$. Then, 
since $u(\bar y_1,\cdot)< v(\bar y_1,\cdot)$,
the maximum $m$ is achieved at some point $z$ such that $z_1<\bar y_1$. We make out three cases:
\begin{enumerate}
\item If $z_1>-\rho$, then we can reproduce the arguments of Imbert and Monneau contained in \cite[Appendix 2]{imbert:hal-01073954} and find a contradiction, 
(note that in the region $-\rho < y_1\le \bar y_1$, the interface $\widetilde \Gamma_{\eta}$
 is made of straight lines, so the arguments in  \cite{imbert:hal-01073954} can be applied in a
 straightforward manner). Alternatively,  it is possible to use  the different methods proposed in 
either \cite{barles2013bellman} or in  \cite{2016arXiv161101977B}. It is also possible to use the arguments contained in
the very recent work of Lions and Souganidis \cite{2017arXiv170404001L}.  We therefore skip this part of the proof.
\item if $z_1=-\rho$ and $z\not\in \widetilde \Gamma_\eta$, then  after a suitable localization, we can apply the now classical arguments of Soner for state constrained boundary conditions and reach a contradiction, see \cite{MR838056,MR861089,MR951880,MR1484411}. We also skip the details for brevity.
\item We will thus focus on the case when the maximum is reached at $z\in  \widetilde \Gamma_\eta$ such that $z_1=-\rho$, because it contains additional difficulties. 
\end{enumerate}
We will make the following steps:
\begin{enumerate}
\item Localize around $z$: in the domain of interest, the interface will be made of only one straight line.
Moreover, it will be convenient to modify the Hamiltonians for large values of $q$, which is always possible since $u$ and $v$ are Lipschitz continuous. 
\item Recall the definition of the vertex test-function of Imbert-Monneau, see \cite{imbert:hal-01073954}, 
which will be named $ G^{\gamma, z}(x,y)$ below. This function will play the role of the penalty term $|x-y|^2$ in the classical arguments consisting of doubling the variables when there is no interface. 
\item Adapt Soner's arguments for state constrained boundary condition to the present case. In the arguments consisting of doubling the variables, we will use Soner's ideas to ensure  that  the viscosity inequalities for $u$ can be written, i.e. that the maximum point $\hat x$  be such that $\hat x_1>-\rho$. 
\end{enumerate}

\paragraph{\bf Step 1: localization and modification of the Hamiltonians}
We are going to localize the problem around $z$: near $z$, $ \widetilde \Gamma_\eta$ coincides with a straight line that we name $\Delta$. Changing the coordinates if necessary, we can assume that $\Delta= \{y: y_2=0\}$, so $z=(-\rho, 0)$. 
It is not restrictive to assume that for $r>0$ small enough, $\Omega^R_\eta \cap\overline B(z,r)=  \{y: y_2>0\}\cap\overline B(z,r)$ and that $\Omega^L_\eta \cap\overline B(z,r)=  \{y: y_2<0\}\cap\overline B(z,r)$. Therefore, the Hamiltonian is $H^R$ in $\{y: y_2>0\}\cap\overline B(z,r)$ and $H^L$ in $\{y: y_2<0\}\cap\overline B(z,r)$, so it does not depend on $y_1$. In what follows, we will always suppose that $r>0$ is small enough so that $B(z,r) \subset \{y_1<\bar y_1\}$ and that we are in the situation described above. \\
Moreover, noting that $u$ and $v$ are both Lipschitz continuous with a constant $L$ which may depend on $p_2$ but not on $\rho$ and $\eta$, we can modify the Hamiltonians $H^L$ and $H^R$ in such a way:
\begin{itemize}
\item  $q\mapsto H^L( p_2e_2 + q)$ and $q\mapsto H^R( p_2e_2 + q)$ are kept unchanged in the ball $|q|\le 2L$ 
\item $q\mapsto H^L( p_2e_2 + q)$ and $q\mapsto H^R( p_2e_2 + q)$ become second order polynomials in $q$ in a neighborhood of 
$|q|=+\infty$
\end{itemize}
For that, it is enough to replace $H^i$ by $q\mapsto \max( H^i( p_2e_2 + q),  a |q|^2 -b ) $ for well chosen positive constants $a$ and $b$.\\
 Let us name  $\HH ^i$, $i=L,R$ the  modified  Hamiltonians.
With the usual notations, we see that $v$ is a  supersolution of 
\begin{equation}
\label{eq:62}
\left\{
    \begin{array}[c]{lll}
    \HH^L(Dv(y))&\ge \lambda_\rho( p_2)&\hbox{ if } y_1\ge -\rho,\; y_2<0,   \hbox{ and } y\in \overline B(z,r),\\
    \HH^R(Dv(y))&\ge \lambda_\rho( p_2)&\hbox{ if } y_1\ge -\rho,\; y_2>0,   \hbox{ and } y\in \overline B(z,r),\\
    \max \{ \HH^{+,2,L}(Dv^L(y)),\HH^{-,2,R} (Dv^R(y))\}&\ge \lambda_\rho( p_2)&\hbox{ if } y_1\ge -\rho,\; y_2=0,   \hbox{ and } y\in \overline B(z,r),
  \end{array}\right.
\end{equation}
and that $u$ is a subsolution of 
\begin{equation}
\label{eq:64}
\left\{
    \begin{array}[c]{lll}
    \HH^L(Du(y))&\le \lambda_\rho( p_2)-\delta&\hbox{ if } y_1> -\rho,\; y_2<0,   \hbox{ and } y\in \overline B(z,r),\\
    \HH^R(Du(y))&\le \lambda_\rho( p_2)-\delta&\hbox{ if } y_1> -\rho,\; y_2>0,   \hbox{ and } y\in \overline B(z,r),\\
    \max \{ \HH^{+,2,L}(Du^L(y)),\HH^{-,2,R} (Du^R(y))\}&\le \lambda_\rho( p_2)-\delta&\hbox{ if } y_1> -\rho,\; y_2=0,   \hbox{ and } y\in \overline B(z,r).
  \end{array}\right.
\end{equation}
 For brevity, we make an abuse of notation and rewrite the  three inequalities in (\ref{eq:62}) and (\ref{eq:64}) as follows:
\begin{equation}
  \label{eq:65}
  \begin{split}
    \HH(y_2, Dv)&\ge \lambda_\rho( p_2) \quad \hbox{for } -\rho\le y_1   \hbox{ and } y\in \overline B(z,r),\\
    \HH(y_2, Du)&\le \lambda_\rho( p_2)-\delta \quad \hbox{for } -\rho <y_1  \hbox{ and } y\in \overline B(z,r),
  \end{split}
\end{equation}
where  for any $y\in \R^2$,
\begin{eqnarray}
  \label{eq:81}
\HH(y_2, p)= \HH^R (p)\quad  \hbox{ if }y_2>0,\\
\label{eq:82}  \HH(y_2, p)= \HH^L (p)\quad  \hbox{ if }y_2<0,\\
\label{eq:83}  \HH(y_2, (p^L, p^R))=  \max \{ \HH^{+,2,L}(p^L),\HH^{-,2,R} (p^R)\} \hbox{ if }y_2=0.
\end{eqnarray}

\paragraph{\bf Step 2: the test-function of Imbert-Monneau}
Following  \cite[Theorem 3.1]{imbert:hal-01073954}, we are going to use the so-called vertex test-function at $z$: for $\gamma$,  $0<\gamma<1$,  there exists a function $G^{\gamma, z}: \R^2\times \R^2\to \R$ with the following properties:
\begin{enumerate}
\item{(Regularity)}
  \begin{displaymath}
    G^{\gamma, z}\in \cC(\R^2\times \R^2) \hbox{ and } \left\{
      \begin{array}[c]{rcl}
        G^{\gamma, z}(X,\cdot)\in \cR \quad &\hbox{for all }& X\in \R^2,\\
        G^{\gamma, z}(\cdot,Y)\in \cR \quad &\hbox{for all }& Y\in \R^2
      \end{array}
\right.
  \end{displaymath}
where $\cR$ is the set of continuous functions on $\R^2$ whose restrictions to $\R\times[0,\pm\infty)$ are $\cC^1$.
If $f\in \cR$ and $x\in \Delta$, $D f(x)$  denotes the pair   $(D f^L(x), D f^R(x)\in \R^2 \times \R^2$.
\item{(Bound from below)} $G^{\gamma, z}\ge 0=G^{\gamma, z}(z,z)$
\item{(Compatibility condition on the diagonal)} For all $X\in \R^2$,
  \begin{equation}
    \label{eq:66}
0\le G^{\gamma, z}(X,X)=G^{\gamma, z}(X,X)-G^{\gamma, z}(z,z)\le \gamma
  \end{equation}
\item{(Compatibility condition on the gradients)} For all $X,Y\in \R^2$ and $K>0$ with $|X-Y|\le K$,
  \begin{equation}
    \label{eq:67}
\HH(Y_2, -D_YG^{\gamma, z}(X,Y))-\HH(X_2, D_XG^{\gamma, z}(X,Y))\le \omega_{C_K}(\gamma C_K)
  \end{equation}
with $C_K$ given in (\ref{eq:69}) below and $\omega_{C_K}$ is a modulus of continuity defined on $[0,C_K]$.
Here we have used the notations given in (\ref{eq:81})-(\ref{eq:83}), so if $Y_2=0$, $D_YG^{\gamma, z}(X,Y)$ 
is a pair of vectors in $\R^2$ with the same first component.
\item{(Superlinearity)} There exists $g: [0,+\infty)\to \R$ nondecreasing and such that for all $X,Y\in \R^2$,
  \begin{equation}\label{eq:68}
    g(|X-Y|)\le G^{\gamma, z}(X,Y)
  \end{equation}
and $a\mapsto g(a)$ can be chosen to be quadratic  ($g(a) = c_1 a^2 +c_2$) for $a$ large enough. 
\item{(Gradient bounds)} For all $K>0$, there exists $C_K>0$ independent of $\gamma$, such that for all   $X,Y\in \R^2$
  \begin{equation}
    \label{eq:69}
|X-Y|\le K \quad \Rightarrow\quad | G^{\gamma, z}_X(X,Y)| + | G^{\gamma, z}_Y(X,Y)| \le C_K .
  \end{equation}
\end{enumerate}
\begin{remark}
  \label{sec:proof-comp-result}
The fact that $g(a)$ can be chosen quadratic for $a$ large enough comes from the fact that $\HH^i(q) $ are second order polynomials in $|q|$ for $|q|$ large enough, see the proof of Proposition 3.3, step 4, in \cite[\S 3.4]{imbert:hal-01073954}. 
\end{remark}

\paragraph{\bf Step 3: doubling the variables}
Let us introduce 
\begin{equation}\label{eq:70}
  m_{\tau,\gamma}= \max_{ 
    \begin{array}[c]{l}
x,y \in \overline{B(z,r)},\\
-\rho\le x_1,y_1
    \end{array}
} \left(u(x)-v(y)- G_\tau^{\gamma,z}(x, y + \kappa(\tau) e_1) -\phi(x) \right),
\end{equation}
where
\begin{equation}\label{eq:71}
   G_\tau^{\gamma,z}(x, y)=  \tau G^{\gamma, z} \left(z+ \frac {x-z} \tau, z+\frac {y-z} \tau\right)
\end{equation}
and 
\begin{equation}
  \label{eq:72}
\phi(x)= \frac 1 2 |x-z|^2.
\end{equation}
Finally $\kappa(\tau)$ is a  power of $\tau$ with a positive exponent that will be chosen later.
\\
For $\tau$ small enough, taking $x= z+\kappa(\tau)e_1$ and $y=z$, we see that
\begin{displaymath}
  \begin{split}
  m_{\tau,\gamma}&\ge u(z+\kappa(\tau)e_1)-v(z)- G_\tau^{\gamma,z}(z+\kappa(\tau)e_1, z + \kappa(\tau) e_1) -\phi(z+\kappa(\tau)e_1)     \\
  & \ge u(z)-v(z)  -L \kappa(\tau) -\tau \gamma  - \frac 1 2 \kappa^2(\tau)\\
  & = m  -L \kappa(\tau) -\tau \gamma  - \frac 1 2 \kappa^2(\tau)
  \end{split}
\end{displaymath}
Hence, there exists $0<\bar \gamma<1$ and $0<\bar \tau$ such that for all $0<\gamma<\bar \gamma$ and $0<\tau <\bar \tau$,
$m_{\tau,\gamma}\ge \frac m 2 >0$.\\
On the other hand,
\begin{displaymath}
  \begin{split}
    u(x)-v(y)- G_\tau^{\gamma,z}(x, y + \kappa(\tau) e_1) -\phi(x)     
&\le u(y)-v(y)+ L|x-y|-  G_\tau^{\gamma,z}(x, y + \kappa(\tau) e_1) -\phi(x) \\
&\le m+ L|x-y|-  G_\tau^{\gamma,z}(x, y + \kappa(\tau) e_1) -\phi(x).
\end{split}
\end{displaymath}
Therefore, if $\hat x$ and $\hat y$ achieve the maximum in (\ref{eq:70}), then
\begin{displaymath}
  G_\tau^{\gamma,z}(\hat x, \hat y + \kappa(\tau) e_1) +\phi(\hat x)\le \tau \gamma +  L \kappa(\tau) + \frac 1 2 \kappa^2(\tau) + L|\hat x-\hat y|.
\end{displaymath}
From (\ref{eq:68}), this implies that 
\begin{equation}
  \label{eq:63}
  \begin{split}
\tau  g\left(\frac {|\hat x -\hat y - \kappa(\tau) e_1|}\tau\right) +\phi(\hat x)& \le \tau\gamma +  L \kappa(\tau) + \frac 1 2 \kappa^2(\tau)+ L|\hat x-\hat y|   \\
&
\le \tau \gamma +  2L \kappa(\tau) + \frac 1 2 \kappa^2(\tau)+ L\tau  \frac {|\hat x-\hat y- \kappa(\tau) e_1|} \tau. 
  \end{split}
\end{equation}
Using the superlinear behavior of $g$ at infinity, we see that there exists a constant $C>0$ independent of $\gamma$ and $\tau$ such that
\begin{displaymath}
  \tau  g\left(\frac {|\hat x -\hat y - \kappa(\tau) e_1|}\tau\right) \le C.
\end{displaymath}
If for a subsequence still called $\tau$, $|\hat x -\hat y - \kappa(\tau) e_1|>0$, then,  setting $d_{\gamma,\tau} = |\hat x -\hat y - \kappa(\tau) e_1|$,
 the latter inequality can be written
\begin{displaymath}
  \frac \tau  {d_{\gamma,\tau}} g\left(\frac {d_{\gamma,\tau}}\tau\right) \le
\frac  C{d_{\gamma,\tau}}.
\end{displaymath}
From the quadratic behavior of  $g$ away from the origin, we know that there exist two positive constants $D$ and $c $ such that
$g(d)\ge c d^2$ for $d>D$.   
If $d_{\gamma,\tau}  /\tau>D$, then  $\frac  C{d_{\gamma,\tau}}\ge \frac \tau  {d_{\gamma,\tau}} g\left(\frac {d}\tau\right) \ge c  \frac {d_{\gamma,\tau}} \tau$. 
%  implies that  $ c \frac d \tau \ge  $ 
% for some positive constant  $c$.
% So either $\frac  C{d_{\gamma,\tau}}\ge cD$ or $d_{\gamma,\tau}/\tau\le D$. 
We can choose $D=1/\sqrt{\tau}$ for $\tau$ small enough, which yields that
$d_{\gamma,\tau}$ is bounded by a quantity of the order of $\sqrt{\tau}$.
\\
We have proved that 
\begin{enumerate}
\item $m_{\tau,\gamma}>m/2>0$ for all $0<\gamma<\bar \gamma$ and $0<\tau <\bar \tau$
\item $|\hat x -\hat y - \kappa(\tau) e_1|\le C\sqrt{\tau}$, for a positive constant $C$  independent of $\gamma$, $0<\gamma<\bar \gamma$
\item $\lim_{\tau\to 0}|\hat x -\hat y| =0$,  uniformly in $0<\gamma<\bar \gamma$
\item From (\ref{eq:63}), we see that $\lim_{(\gamma, \tau)\to (0,0)} \hat x= z$.  Hence, for $\tau$ and $\gamma$ small enough,
$\hat x \in B(z,r)$ and $\hat y \in B(z,r)$.
\end{enumerate}
Moreover, choosing $\kappa(\tau)=\tau^{1/3}$ for example, we find that $\hat x_1> -\rho$ for $\tau$ small enough.\\
This allows us to write the following viscosity inequalities:
\begin{eqnarray}
  \label{eq:73}
\HH(\hat x_2, p_X^{\tau,\gamma} + \hat x -z ) &\le& \lambda_\rho(p_2)-\delta,
\\
\HH(\hat y_2, p_Y^{\tau,\gamma}  ) &\ge & \lambda_\rho(p_2),
\label{eq:74}
\end{eqnarray}
(with the notations introduced in (\ref{eq:81})-(\ref{eq:83})),
where 
\begin{eqnarray}\label{eq:75}
   p_X^{\tau,\gamma}&=& G_X^{\gamma,z}\left(  z+\frac {\hat x-z} \tau,z +\frac {\hat y +\kappa(\tau) e_1 -z} \tau  \right),\\
\label{eq:76}   p_Y^{\tau,\gamma}&=& -G_Y^{\gamma,z}\left( z+\frac {\hat x-z} \tau,z+\frac {\hat y +\kappa(\tau) e_1-z} \tau  \right).
\end{eqnarray}
% Then (\ref{eq:73}) implies that there exists a constant $C$ independent of $\tau$ and $\gamma$ such that
% \begin{eqnarray}\label{eq:77}
%   |p_X^{\tau,\gamma} | \le C \quad \hbox{ if } \hat x\not\in  \Gamma,\\ 
% \label{eq:78}
%   |p_{X,1}^{\tau,\gamma} | + \max(0,-  p_{X,2}^{\tau,\gamma} )  \le C \quad \hbox{ if } \hat x \in  \Gamma.
% \end{eqnarray}
\begin{enumerate}
\item If $\hat x\not \in \Delta$, then the coercivity of the Hamiltonians $\HH^L$ and $\HH^R$ implies that, for  a constant $C$ independent of $\tau$ and $\gamma$,
  \begin{equation}
  \label{eq:77}    |p_X^{\tau,\gamma} | \le C.
  \end{equation}
% We then use (\ref{eq:77}) and (\ref{eq:69}) to obtain that 
% $| p_Y^{\tau,\gamma}| \le \frac C {\sqrt{\tau}} $ for $\tau$ small enough.
%With this piece of information,  s
Subtracting (\ref{eq:74}) and (\ref{eq:73}) yields that 
\begin{equation}
  \label{eq:100}
  \HH(\hat y_2, p_Y^{\tau,\gamma}  )- \HH(\hat x_2, p_X^{\tau,\gamma} + \hat x -z )\ge \delta,
\end{equation}
which is equivalent to
\begin{displaymath}
  \HH( \frac {\hat y_2 } \tau , p_Y^{\tau,\gamma}  )- \HH(\frac {\hat x_2}\tau, p_X^{\tau,\gamma} + \hat x -z )\ge \delta.
\end{displaymath}
Note that $\frac {\hat y_2 } \tau $ is also the second component of  $z+\frac {\hat y -z +\kappa(\tau)e_1} \tau$
and that  $\frac {\hat x_2 } \tau $ is the second component of   $z+\frac {\hat x -z } \tau$.
Then, using (\ref{eq:67}) and the fact that $|\hat x -\hat y - \kappa(\tau) e_1|\le C\sqrt{\tau}$, 
we see that 
\begin{displaymath}
  \HH(\frac {\hat y_2 } \tau, p_Y^{\tau,\gamma}  )-  \HH(\frac {\hat x_2 } \tau, p_X^{\tau,\gamma}  )\le \omega_{ \frac C {\sqrt{\tau}}} (  \frac C {\sqrt{\tau}} \gamma)
\end{displaymath}
or equivalently,
\begin{equation}
  \label{eq:49}
  \HH(\hat y_2 , p_Y^{\tau,\gamma}  )-  \HH(\hat x_2 , p_X^{\tau,\gamma}  )\le \omega_{ \frac C {\sqrt{\tau}}} (  \frac C {\sqrt{\tau}} \gamma).
\end{equation}
Adding and subtracting $ \HH(\hat x_2, p_X^{\tau,\gamma}  )$ in (\ref{eq:100}) and using (\ref{eq:49})
yields
\begin{displaymath}
   \HH(\hat x_2, p_X^{\tau,\gamma}  )- \HH(\hat x_2, p_X^{\tau,\gamma} + \hat x -z ) +\omega_{ \frac C {\sqrt{\tau}}} (  \frac C {\sqrt{\tau}} \gamma) \ge \delta.
\end{displaymath}
Using the properties of the Hamiltonians and (\ref{eq:77}), we get that,
 for some constant  $\tilde C$ independent of $\tau$ and $\gamma$,
\begin{displaymath}
\tilde C |\hat x-z|
 +\omega_{ \frac C {\sqrt{\tau}}} (  \frac C {\sqrt{\tau}} \gamma) \ge \delta.
\end{displaymath}
This yields a contradiction by having $\gamma\to 0$ then $\tau\to 0$.
\item  If $\hat x \in \Delta$ or equivalently $\hat x_2=0$, we see that $\hat x- z$ is colinear to $e_1$ and that
$ \max \{ \HH^{+,2,L}( (p_X^{\tau,\gamma})^L + \hat x- z),\HH^{-,2,R}  ((p_X^{\tau,\gamma})^R+ \hat x- z)\}\le \lambda_\rho( p_2)-\delta$.
This implies 
that for a constant $C>0$ independent of $\tau$ and $\gamma$,
\begin{equation}
   \label{eq:78}
   |p_{X,1}^{\tau,\gamma} | + \max\left(0,-  (p_{X,2}^{\tau,\gamma})^R \right) +  \max\left(0, (p_{X,2}^{\tau,\gamma})^L \right)    \le C .
\end{equation}
where $p_{X,1}^{\tau,\gamma}$ stands for the first coordinate of both $(p_{X}^{\tau,\gamma})^L$ and $(p_{X}^{\tau,\gamma})^R$.
Then using the arguments of Imbert and Monneau in \cite[\S 5.5]{MR3621434},
we can find a constant $K$  independent of $\tau$ and $\gamma$ such that 
\begin{displaymath}
  \HH(0, \bar p_X^{\tau,\gamma}  )=\HH(0, p_X^{\tau,\gamma}  ) \quad \hbox{and}\quad
  \HH(0, \bar p_X^{\tau,\gamma}  + \hat x -z )=\HH(0, p_X^{\tau,\gamma} + \hat x -z  ),
\end{displaymath}
where $ \bar p_{X,1}^{\tau,\gamma}=p_{X,1}^{\tau,\gamma}$,
$ (\bar p_{X,2}^{\tau,\gamma})^R =\min\left( K,(p_{X,2}^{\tau,\gamma})^R\right)$ and
$ (\bar p_{X,2}^{\tau,\gamma})^L =\max\left( -K,(p_{X,2}^{\tau,\gamma})^L\right)$.
Note that we have used the fact that $ \hat x -z $ is colinear to $e_1$ and bounded independently of $\tau$ and $\gamma$.
Since $ |\bar p_{X}^{\tau,\gamma}| \le C$ for a constant $C$ independent of $\tau$ and $\gamma$, there exists a constant $\tilde C$  such that
\begin{equation}
  \label{eq:79}
  \begin{split}
|  \HH(0, p_X^{\tau,\gamma}  )- \HH(0, p_X^{\tau,\gamma} + \hat x -z )| =     
|  \HH(0, \bar p_X^{\tau,\gamma}  )- \HH(0, \bar p_X^{\tau,\gamma} + \hat x -z )| \le \tilde C| \hat x -z|.
  \end{split}
\end{equation}
On the other hand, we have, exactly as above, that
\begin{equation}
  \label{eq:80}
|  \HH(\hat x_2, p_X^{\tau,\gamma}  ) -\HH(\hat y_2, p_Y^{\tau,\gamma}  )|= 
|  \HH( \frac {\hat x_2} \tau , p_X^{\tau,\gamma}  ) -\HH( \frac {\hat y_2} \tau  , p_Y^{\tau,\gamma}  )|
 \le \omega_{ \frac C {\sqrt{\tau}}} (  \frac C {\sqrt{\tau}} \gamma).
\end{equation}
Subtracting (\ref{eq:74}) and (\ref{eq:73}), then using (\ref{eq:79}) and (\ref{eq:80}), and letting $\gamma$ tend to $0$ then $\tau$ tend to $0$ yields the desired contradiction.
\end{enumerate}
\end{proof}

\paragraph{Proof of Proposition~\ref{cor:slopes_omega}}
The proof follows easily from Lemma \ref{SLOPESLemma1} and the local uniform 
convergence of the sequence $\chi_\rho( p_2,\cdot)$ toward $\chi( p_2,\cdot)$, by letting $\rho$ tend $+\infty$ and $\delta$ tend to $0$.
\qed

\paragraph{Proof of Proposition~\ref{cor:control_slopes_W}}
 From Lemma~\ref{lem:rescaling_omega}, we see that $y\mapsto W(p_2,y)$ is Lipschitz continuous w.r.t. $y_1$ and independent of $y_2$, and satisfies 
 \begin{equation}
\label{eq:13}
H^R(\partial_{y_1} W(p_2,y) e_1+p_2e_2)=E^{M,R}( p_2) \quad \hbox{  for a.a. } y_1>\eta. 
 \end{equation}
Consider first the case when $E^{M,R}(p_2)>E_0^R(p_2)$;
from the convexity and coercivity of $H^R$, the observations above yield that almost everywhere in $y$, $\partial_{y_1}W(p_2,y)$ can be either $\Pi^R(p_2)$
(the unique real number such that $H^{+,R}( qe_1+p_2e_2)=E^{M,R}( p_2)$), 
or the unique real number  $q$ (depending on $(p_2)$) such that $ H^{-,R}( qe_1+p_2e_2)=E^{M,R}( p_2)$.  Note that $q<  \Pi^R(p_2)$.
But from Proposition \ref{cor:slopes_omega} and the local uniform convergence of $W_\eta(p_2,\cdot)$ toward $W(p_2,y)$, 
we see that  that for any $y_1>\eta$ and $h_1\ge 0$,
\begin{displaymath}
W(p_2,y+h_1 e_1)-W(p_2,y)\ge \Pi^R(p_2) h_1,
\end{displaymath}
which implies that almost everywhere, $\partial_{y_1}W(p_2,y) \ge  \Pi^R(p_2)>q$. Therefore,  
 $\partial_{y_1}W(p_2,\cdot)= \Pi^R(p_2)$ for almost all $y_1>\eta$.
\\
In the case when  $E^{M,R}(p_2)=E_0^R(p_2)$, we deduce from (\ref{eq:13}) that
 for almost all $y_1>\eta$,   $\overline{\Pi}^R(p_2)\le \partial_{y_1}W(p_2,y) \le \widehat{\Pi}^R(p_2)$. 
\\
We have proved (\ref{cor:control_slopes_W1}). The proof of (\ref{cor:control_slopes_W2}) is identical.
 Finally, (\ref{eq:control_slopes_W_summary}) comes from (\ref{cor:control_slopes_W1}),~(\ref{cor:control_slopes_W2})  and from the fact that $W(p_2,\eta e_1)=0$.
\qed

\noindent{\bf Acknowledgement.} \quad   The work was partially supported by ANR projects ANR-12-BS01-0008-01 and ANR-16-CE40-0015-01.

{\small
\bibliographystyle{amsplain}
\bibliography{homog_interface}
}

\end{document}